\documentclass[11pt,a4paper,twoside,openright]{article}
\usepackage[english]{babel}
\usepackage{amsmath,amssymb, amsfonts,epic}
\usepackage[dvips]{graphicx}

\usepackage{setspace}
\newcommand{\s}{\, \mathrm{s}}
\newcommand{\m}{\, \mathrm{m}}

\newcommand{\mss}{\, \mathrm{m/s^2}}
\newcommand{\rad}{\,\mathrm{rad}}
\newcommand{\km}{\, \mathrm{km}}
\newcommand{\kms}{\, \mathrm{km/s}}
\newcommand{\dd}{\mathrm{d}}

\newcommand{\e}{\varepsilon}
\newcommand{\N}{\mathbb{N}}
\newcommand{\R}{\mathbb{R}}
\newcommand{\PP}{\mathbb{P}}

\newcommand{\esp}{\mathbb{E}}

\newcommand{\Var}{\mathrm{Var}}
\newcommand{\Cov}{\mathrm{Cov}}

\newcommand{\PPst}{\PP_{\theta^*}}

\newcommand{\opst}{\mathrm{o_{\PPst}}}

\newcommand{\oopst}{\mathrm{O_{\PPst}}}

\newtheorem{thm}{Theorem}

\newtheorem{prop}{Proposition}

\newtheorem{assumption}{Assumption}
\newtheorem{lemma}{Lemma}
\newenvironment{proof}[1][Proof] {\textit{#1}}{}

\makeatletter
\renewcommand{\subsection}{\@startsection{subsection}{2}{0mm}{\baselineskip}{.25\baselineskip}{\it}}
\makeatother

\begin{document}

\begin{flushleft}
{\Large \bf Semiparametric regression estimation
using noisy nonlinear non invertible functions of the observations.} \\
\
\\

{\large ELISABETH GASSIAT } \\
{\it \'Equipe Probabilit\'es, Statistique et
 Mod\'elisation,   Universit\'e Paris-Sud 11 and CNRS}\\
 \ \\

{\large BENOIT LANDELLE}\\
{\it \'Equipe Probabilit\'es, Statistique et
 Mod\'elisation,   Universit\'e Paris-Sud 11, CNRS,
and Thales Optronique}
\end{flushleft}

\begin{flushleft}
{ \bf ABSTRACT.
We investigate a semiparametric regression model  where one gets noisy non linear non invertible functions of
the observations. We focus on the application to bearings-only tracking. We first investigate
the least squares estimator and prove its consistency and asymptotic normality under mild assumptions.
We study the semiparametric likelihood process and prove local asymptotic normality of the model.
This allows to define the efficient Fisher information as a lower bound for the asymptotic variance of
regular estimators, and to prove that the parametric likelihood estimator is regular and asymptotically
efficient. Simulations are presented to illustrate our results.
}
\end{flushleft}


\section{Introduction}

In bearings-only tracking (BOT), one gets information about the trajectory of a target only via bearing measurements obtained
by a moving observer. This is a  highly ill-posed problem which requires, so that one be able to propose solutions,
the choice of a trajectory model.
The literature on the subject is very large, and many algorithms have been proposed to track the target, see for instance
\cite{barshalom:li:kirubarajan:2001}, \cite{doucet:defreitas:gordon:2001}, \cite{averbuch:barshalom:dayan:mazor:1998}, \cite{arulampalam:gordon:ristic:2004}.
All these algorithms are designed for particular classes of models for the trajectory of the target. In \cite{landelle:2008}, the author proved that the least squares estimator may be very sensitive to some small deterministic perturbations, in which case the algorithms are highly  non robust. However, it has been also claimed in \cite{landelle:2008} that stochastic perturbations do not
essentially alter
the performances of the estimator. The aim of this paper is to develop an estimation theory for a semiparametric model
that applies to BOT.
The model we study is the following:
\begin{equation}
\label{eq:model}
    \left\{
    \begin{array}{l}
      X_k= S_\theta(t_k) + \varepsilon_k ,\\
      Y_k=\Psi(X_k,t_k)+V_k .
    \end{array}
    \right.
\end{equation}
$(t,\theta) \mapsto S_\theta(t)$
is a known map from $[0,1] \times \Theta$ to $\R^{d}$,  $\Theta$
is the parameter set (in general, a subset of a finite dimensional euclidian space),
$(x,t) \mapsto \Psi (x,t)$ is a known function from $\R^{d}\times [0,1]$ to $\R$, which in general is non invertible, $(t_{k})_{k\in\N}$ is the
sequence of observation times in $[0,1]$, $(\varepsilon_k )_{k\in\N}$ is a sequence of random variables
taking values in $\R^{d}$, $(V_k )_{k\in\N}$ is a sequence of centered i.i.d. random variables
taking values in $\R $,
with known marginal distribution $g(x)dx$, variance $\sigma^{2}$ , and independent
of the sequence $(\varepsilon_k )_{k\in\N}$.
The sequence $(X_k )_{1 \leq k \leq n}$ is not observed.
We aim at estimating $\theta$ using only the observations  $(Y_k )_{1 \leq k \leq n}$.\\
In case of BOT,
$(X_k )_{1 \leq k \leq n}$ is the trajectory of the target, given by its euclidian coordinates at times $( t_k )_{1 \leq k \leq n}$
 ($d=2$),
$S_\theta(\cdot)$ is the parametric trajectory the target is assumed to follow up to some parameter
$\theta$, for instance uniform linear motion, or a sequence of uniform linear and circular motions,
$(\varepsilon_k)_{1 \leq k \leq n}$ is a  noise sequence to take into account the fact that
the model is only an idealization of the true trajectory and to allow stochastic departures of the
trajectory model, and $(V_k)_{1 \leq k \leq n}$ is the observation noise. Since the observer is moving, if
$(O(t) )_{t \in [0,1]}$ is its trajectory, 
the function $\Psi (x,t)$ is the angle, with respect
to some fixed direction, of $x - O(t) $, that is, for $x=(x_1 , x_2 )^T$:
\begin{equation}
\label{eq:psi:bot}
    \Psi(x,t )= \arctan[x_{2}-O_2(t)]/[x_{1}-O_1(t)].
\end{equation}
In such a case, for any $z$ and fixed $t$, the set $\{x:\Psi(x,t )=z\}$ is infinite.
Our aim here is to understand how it is possible to estimate the parameter $\theta$ in model (\ref{eq:model}), what are the
limitations in the statistical performances,  to propose estimation procedures,
to build confidence regions for $\theta$ and to discuss their optimality
under the weakest possible assumptions on the sequence $(\varepsilon_k )_{k\in\N}$.
Indeed, we would like to apply the results to BOT under
realistic assumptions, for which it is not a strong assumption to assume that the observation noise
$(V_k )_{k\in\N}$ consists of i.i.d. random variables with known distribution,
but the trajectory noise
$(\varepsilon_k )_{k\in\N}$ may be quite complicated and unknown. To begin with, we will
assume that the variables $(\varepsilon_k )_{k\in\N}$ are i.i.d. with unknown
distribution.\\
As such, the model may be viewed as a regression model with two variables,  in which one of the variables
is random, is not observed and follows itself a regression model. One could think that it looks
like  an inverse problem, or that the model may be understood as a state space model,
or a mixed effects model, but in a nonstandard way, so that we have not been able to find results
in the literature that apply
to this setting.\\
Throughout the paper, observations $(Y_{k})_{1\leq k\leq n}$
are assumed to follow model (\ref{eq:model}) with true (unknown) parameter $\theta^{*}$ and the observation
times are $t_{k}=\frac{k}{n}$, $k=1,\ldots,n$. All norms $\|\cdot\|$ are euclidian norms.

In Section \ref{sec:ls}, we consider least squares estimation and prove consistency and
asymptotic normality in this setting, see Theorems \ref{theo:ls:cons}
and \ref{theo:ls:clt}. This allows to introduce basic considerations and set some
assumptions. We prove that the results apply to BOT for linear observable trajectory
models and when the trajectory noise has an isotropic distribution,
see Theorem \ref{theo:BOT:1}.
Then, in Section \ref{sec:lik}
we study the likelihood process to set local asymptotic normality
and efficiency in the parametric setting where the density of the noise $(\varepsilon_k )_{k\in\N}$ is known,
and define the efficient Fisher information in the semiparametric setting
where the density of the noise $(\varepsilon_k )_{k\in\N}$ is unknown.
This also gives an estimation criterion which may be used even if the trajectory noise is correlated.
In Section \ref{sec:dep}, we propose strategies for semiparametric estimation
and discuss possible extension of the results to possibly dependent trajectory noise
 $(\varepsilon_k )_{k\in\N}$.
Section \ref{sec:simu} is devoted to simulations.
In each section, particular attention is given to the application of the results to BOT.

\section{Least squares estimation}
\label{sec:ls}
In sections \ref{sec:ls} and \ref{sec:lik}  we will use
\begin{assumption}
\label{ass:bruit1}
$(\varepsilon_k )_{k\in\N}$ is a sequence of i.i.d. random variables.
\end{assumption}
To be able to obtain a consistent estimator of $\theta$, we require that, in the absence of noise
(both observation noise and trajectory noise), the observation at all times is sufficient to retrieve the
parameter. We thus introduce
\begin{assumption}
\label{ass:obs}
If $\theta\in\Theta$
is such that
$\Psi(S_{\theta}(t),t)=\Psi(S_{\theta^{*}}(t),t)$ a.e. for all $t\in[0,1]$, then
$\theta=\theta^{*}$.
\end{assumption}
This is the observability assumption.\\
If the observation noise is centered, in the absence of trajectory noise,
the fact that only $\Psi(S_{\theta}(t),t)$ is observed with additive noise
is not an obstacle to the estimation of $\theta$ under Assumption \ref{ass:obs}. But with trajectory noise,
only the distribution of $ \Psi(S_{\theta}(t)+\varepsilon_1,t)$ may be retrieved from noisy data. In case the marginal distribution
of the $\varepsilon_k$'s is known, this may be enough, but in case it is unknown, one has to be aware of some link
between the distribution of $ \Psi(S_{\theta}(t)+\varepsilon_1,t)$ and $\theta$. We thus introduce the following assumption, which
will be proved to hold in some BOT situations.
\begin{assumption}
\label{ass:esp}
For all $t\in [0,1]$, for all $\theta \in \Theta$,
        \begin{equation*}
            \esp \{ \Psi[S_\theta(t)+\e_{1} ,t] \} =  \Psi[S_\theta(t),t].
        \end{equation*}
\end{assumption}
Let us now define the least  squares criterion and the least  squares estimator (LSE) by
\begin{eqnarray*}
    M_n(\theta) &=&\frac{1}{n} \sum_{k=1}^{n} \left( Y_k - \Psi[S_\theta(t_k),t_k] \right)^2 , \\
    \overline{\theta}_n  &=&  \arg \min_{\theta \in \Theta} M_n(\theta),
\end{eqnarray*}
where $\arg \min_{\theta \in \Theta} M_n(\theta)$ is any minimizer of $M_{n}$.
\subsection{Consistency}
\label{subsec:ls:cons}
We assume that $\Theta$ is a compact subset of $\R^{m}$, and we will use
\begin{assumption}
\label{ass:consi}
$t \mapsto \esp \left( \Psi[S_{\theta^{*}}(t)+\e_{1},t] \right)^2 $ defines a finite continuous function
on $[0,1]$,
$\sup_{t \in [0,1]} \esp \{\left( \Psi[S_{\theta^{*}}(t)+\e_{1},t] \right)^2 1_{\left( \Psi[S_{\theta^{*}}(t)+\e_{1},t] \right)^2 >M}\}$ tends to $0$ as $M$ tends to infinity,
 and $(t,\theta) \mapsto  \Psi[S_{\theta}(t) ,t]  $ defines a finite continuous function
on $[0,1]\times \Theta$.
\end{assumption}
\begin{thm}
\label{theo:ls:cons}
Under assumptions \ref{ass:bruit1} , \ref{ass:obs}, \ref{ass:esp} and \ref{ass:consi},
$\overline{\theta}_n$ converges in probability to $\theta^*$ as $n$ tends to infinity.
\end{thm}
The proof is a consequence of general results in $M$-estimation.
We begin with a simple Lemma:
\begin{lemma}
\label{lem:1}
Under Assumption \ref{ass:bruit1}, if $F(\cdot,\cdot)$ is a real function
on $\R^{d}\times [0,1]$ such that
$\sup_{t\in [0,1]}\esp |F(\e_{1},t)|$ is finite,
$\lim_{M\rightarrow +\infty}\sup_{t\in [0,1]}\esp \{|F(\e_{1},t)|1_{ |F(\e_{1},t)|>M}\}=0$,
and $\esp F(\e_{1},\cdot)$ is Riemann-integrable,
then
$$\frac{1}{n}\sum_{k=1}^{n} F\left( \e_{k},t_{k}\right)
$$
converges in probability to $\int_{0}^{1}\esp F\left( \e_{1},t\right)dt$ as $n$ tends to infinity.
\end{lemma}
\begin{proof}
\\
First of all, by the integrability assumption,
$$\frac{1}{n}\sum_{k=1}^{n} \esp F\left( \e_{k},t_{k}\right)
$$
converges to $\int_{0}^{1}\esp F\left( \e_{1},t\right)dt$ as $n$ tends to infinity. Then
\begin{eqnarray*}
\frac{1}{n}\sum_{k=1}^{n} \left[F\left( \e_{k},t_{k}\right)-\esp F(\e_{k},t_{k})\right]&=&
\frac{1}{n}\sum_{k=1}^{n}\left[ F\left( \e_{k},t_{k}\right)1_{ |F(\e_{k},t_{k})|>M}-\esp \{F(\e_{k},t_{k})1_{ |F(\e_{k},t_{k})|>M}\}\right]\\
&&+
\frac{1}{n}\sum_{k=1}^{n}\left[ F\left( \e_{k},t_{k}\right)1_{ |F(\e_{k},t_{k})|\leq M}-\esp \{F(\e_{k},t_{k})1_{ |F(\e_{k},t_{k})|\leq M}\}\right].
\end{eqnarray*}
The variance of the second term is upper bounded by $2\frac{M^{2}}{n}$ so that the second term
tends to $0$ in probability as $n$ tends to infinity, and
the absolute value of the first term has expectation upper bounded by
$2\sup_{t\in [0,1]}\esp \{|F(\e_{1},t)|1_{ |F(\e_{1},t)|>M}\}$, which may be made smaller than any positive
$\epsilon$ for big enough $M$,
which proves the lemma.
\end{proof}
\\ Define now
\begin{equation*}
    M(\theta) =  \int_{0}^{1} \esp \left( \Psi[S_{\theta^{*}}(t)+\e_{1},t]- \Psi[S_{\theta}(t),t] \right)^2 \,dt + \sigma^2.
\end{equation*}
Direct calculations yield
\begin{multline*}
    M(\theta)-M(\theta^*) \\
    =\int_{0}^{1} \esp \left(
    \left\{ \Psi[S_{\theta^*}(t)+\e_1 ,t]- \Psi[S_{\theta}(t),t] \right\}^2 -
    \left\{ \Psi[S_{\theta^*}(t)+\e_1,t]- \Psi[S_{\theta^*}(t),t] \right\}^2 \right)\,dt \\
     =\int_{0}^{1} \left\{  \Psi[S_{\theta^*}(t),t] -  \Psi[S_{\theta}(t),t] \right\} \times
    \left\{ 2 \esp \left( \Psi[S_{\theta^*}(t)+\e_1,t]\right)- \Psi[S_{\theta^*}(t),t]- \Psi[S_{\theta}(t),t] \right\} \, dt.
\end{multline*}
By Assumption \ref{ass:esp}, it follows that
\begin{equation*}
    M(\theta)-M(\theta^*) =  \int_{0}^{1}  \left\{ \Psi[S_\theta(t),t]- \Psi[S_{\theta^*}(t),t] \right\}^2 \, dt
\end{equation*}
so that $M(\theta)$ has a unique minimum at $\theta^{*}$ by Assumption \ref{ass:obs}.
Also, under Assumption \ref{ass:consi},
$\theta \mapsto M(\theta)$ is uniformly continuous from $\Theta$ to $\R$.\\
Now, for any $\theta$,
\begin{eqnarray*}
M_{n}(\theta)&=&\frac{1}{n}\sum_{k=1}^{n} V_{k}^{2}+
\frac{2}{n}\sum_{k=1}^{n} V_{k}\left(\Psi[S_{\theta^*}(t_k)+\e_{k},t_k]- \Psi[S_\theta(t_k),t_k]\right)\\
&&+\frac{1}{n}\sum_{k=1}^{n} \left( \Psi[S_{\theta^*}(t_k)+\e_{k},t_k] - \Psi[S_\theta(t_k),t_k] \right)^2.
\end{eqnarray*}
$\frac{1}{n}\sum_{k=1}^{n} V_{k}^{2}$ converges in probability to $\sigma^{2}$;
the variance of
$\frac{1}{n}\sum_{k=1}^{n} V_{k}\left(\Psi[S_{\theta^*}(t_k)+\e_{k},t_k]- \Psi[S_\theta(t_k),t_k]\right)$ is $\frac{\sigma^{2}}{n^{2}}\sum_{k=1}^{n}\esp  \left(\Psi[S_{\theta^*}(t_k)+\e_{k},t_k]- \Psi[S_\theta(t_k),t_k]\right)^2$,
which converges to $0$,
so that \\
$\frac{2}{n}\sum_{k=1}^{n} V_{k}\left(\Psi[S_{\theta^*}(t_k)+\e_{k},t_k]- \Psi[S_\theta(t_k),t_k]\right)$
converges in probability to $0$;
and applying Lemma \ref{lem:1},
$\frac{1}{n}\sum_{k=1}^{n} \left( \Psi[S_{\theta^*}(t_k)+\e_{k},t_k] - \Psi[S_\theta(t_k),t_k] \right)^2$
converges in probability to \\
$\int_{0}^{1} \esp \left( \Psi[S_{\theta^{*}}(t)+\e_{1},t]- \Psi[X_{\theta}(t),t] \right)^2 \,dt$. Thus for any $\theta\in\Theta$, $M_{n}(\theta)$ converges in probability to $M(\theta)$.\\
Using the compacity of $\Theta$ and the second part of Assumption \ref{ass:consi}, it is possible to strengthen this pointwise convergence to a uniform one:
\begin{equation}
\label{eq:mn}
    \sup_{\theta \in \Theta} |M_n(\theta)-M(\theta)| = \opst(1).
\end{equation}
Indeed, for any $\theta_{1}$ and $\theta_{2}$ in $\Theta$,
\begin{multline*}
M_n(\theta_{1})-M_{n}(\theta_{2})\\
=\frac{1}{n}\sum_{k=1}^{n}\left(2\Psi[S_{\theta^*}(t_k)+\e_{k},t_k]-
\Psi[S_{\theta_{1}}(t_k),t_k]-\Psi[S_{\theta_{2}}(t_k),t_k]\right)
\left(\Psi[S_{\theta_{2}}(t_k),t_k]-\Psi[S_{\theta_{1}}(t_k),t_k]\right)\\
+\frac{2}{n}\sum_{k=1}^{n}V_{k}\left(\Psi[S_{\theta_{2}}(t_k),t_k]- \Psi[S_{\theta_{1}}(t_k),t_k]\right)
\end{multline*}
so that for any $\delta>0$,
$$
\sup_{\|\theta_{1}-\theta_{2}\|\leq \delta}\vert M_n(\theta_{1})-M_{n}(\theta_{2})\vert
\leq  \omega (\delta ) \left[\frac{1}{n}\sum_{k=1}^{n}\left(2\vert \Psi[S_{\theta^*}(t_k)+\e_{k},t_k]\vert
+ 2 \sup_{\theta, t}\vert \Psi[S_{\theta }(t ),t ]\vert +2|V_{k}|\right)\right]
$$
where $\omega (\cdot)$ is the uniform modulus of continuity of
$(t,\theta) \mapsto \Psi[S_{\theta}(t) ,t]  $. The right-hand side of the inequality
converges in probability by Lemma \ref{lem:1} to a constant times $\omega (\delta )$, so that
equation (\ref{eq:mn}) follows from compacity of $\Theta$.
Theorem \ref{theo:ls:cons} now follows from \cite{vandervaart:1998} Theorem 5.7.
\subsection{Asymptotic normality}
\label{subsec:ls:as}
Asymptotic normality of the least squares estimator will follow using usual arguments under further
regularity assumptions.
\begin{assumption}
\label{ass:lsclt}
There exists a neighborhood $U$ of $\theta^*$ such that for all $t\in[0,1]$,
$\theta \mapsto  \Psi[S_{\theta}(t) ,t]$ possesses two derivatives on $U$ that are continuous as functions
of $(\theta,t)$ over $U\times [0,1]$.
\end{assumption}
If $\theta \mapsto F$ is a twice differentiable function, let $\nabla_{\theta}F (\theta')$ denote the gradient of $F$
at $\theta'$,
and $D^{2}_{\theta}F (\theta')$ the hessian of $F$ at $\theta'$.
Define for $\theta\in U$:
\begin{eqnarray*}
I_R(\theta)&=& \int_{0}^{1} \nabla_\theta \Psi[S_\theta(t),t]  \nabla_\theta \Psi[S_\theta(t),t]^{T} \, dt \\
I_\Psi(\theta) &=& \int \esp \left\{ \Psi[S_{\theta^* }(t)+\e_{1},t] -\Psi[S_{\theta }(t),t]
    \right\}^2  \nabla_\theta \Psi[S_\theta(t),t]  \nabla_\theta \Psi[S_\theta(t),t]^{T}  dt.
\end{eqnarray*}
Then:
\begin{thm}
\label{theo:ls:clt}
Under Assumptions \ref{ass:bruit1} , \ref{ass:obs}, \ref{ass:esp}, \ref{ass:consi} and \ref{ass:lsclt},
if $I_R(\theta^*)$ is non singular,
\begin{multline*}
    \sqrt{n} (\overline{\theta}_n - \theta^*)  =\\
     I_R(\theta^*)^{-1} \frac{1}{\sqrt{n}} \sum_{k=1}^{n}
    \left\{ \Psi[S_{\theta^*}(t_k)+\e_k,t_k] -\Psi[S_{\theta^*}(t_k),t_k] +V_k
    \right\} \nabla_\theta \Psi[S_{\theta^*}(t_k),t_k] + \opst(1).
\end{multline*}
In particular,
$\sqrt{n}(\overline{\theta}_n - \theta^*)$ converges in distribution to ${\cal N}\left(0,I_M^{-1}(\theta^*) \right)$
where
$$
I_M^{-1}(\theta^*)=  I_R^{-1}(\theta^*)\left[I_\Psi(\theta^*)+\sigma^2 I_R(\theta^*) \right] I_R^{-1}(\theta^*).$$
\end{thm}
Let us notice that, for a null sequence $(\varepsilon_k)_{k\in \N}$, we retrieve the usual Fisher information matrix for the parametric regression model.

The proof follows Wald's arguments. On the set $(\overline{\theta}_n \in U )$, which has probability
tending to $1$ according to Theorem \ref{theo:ls:cons}:
\begin{equation*}
    \nabla_\theta M_n(\overline{\theta}_n) = 0 = \nabla_\theta M_n(\theta^*) + \int_{0}^{1} D^{2}_{\theta} M_n[\theta^*+s(\overline{\theta}_n-\theta^*)] \, ds \:(\overline{\theta}_n-\theta^*).
\end{equation*}
Direct calculations yield for any $\theta\in U$
\begin{equation*}
    \label{eq:Grad_M_n}
    \nabla_\theta M_n(\theta) = -\frac{2}{n}  \sum_{k=1}^{n}\left\{ \Psi[S_{\theta^*}(t_k)+\e _k,t_k]+V_k - \Psi[S_{\theta} (t_k),t_k] \right\}
    \nabla_\theta \Psi[S_{\theta}(t_k),t_k],
\end{equation*}
and
\begin{multline}
\label{eq:Hessian_M_n}
    D^{2}_{\theta}  M_n(\theta) = \frac{2}{n} \sum_{k=1}^{n} \nabla_\theta \Psi[S_\theta(t_k),t_k] \nabla_\theta \Psi[S_\theta(t_k),t_k]^{T} \\
    - \frac{2}{n} \sum_{k=1}^{n}
    \left\{ \Psi[S_{\theta^*}(t_k)+\e_k,t_k]+V_k - \Psi[S_\theta(t_k),t_k] \right\}
    D^{2}_{\theta} \Psi[S_\theta(t_k),t_k].
\end{multline}
Notice that, using Assumption \ref{ass:esp}, $\nabla_\theta M_n(\theta^*)$ is a centered random variable, and that, using
Assumptions \ref{ass:consi}, \ref{ass:lsclt},  the variance of $\nabla_\theta M_n(\theta^*)$
converges to $4\left[I_\Psi(\theta^*)+\sigma^2 I_R(\theta^*) \right]$ as $n\rightarrow +\infty$.
Also using
Assumptions \ref{ass:esp}, \ref{ass:consi}, \ref{ass:lsclt}, and applying Lemma \ref{lem:1},
$D^{2}_{\theta}  M_n(\theta)$ converges in probability to $2I_{R}(\theta)$
as $n\rightarrow +\infty$.\\
Using Assumption \ref{ass:lsclt}, there exists an increasing function $\omega$ satisfying $\lim_{\delta \to 0} \omega(\delta)=0$ such that, for all $(\theta,\theta') \in U^2$ with $\|\theta-\theta'\| \leq \delta$,
\begin{equation*}
  \left\|  D^{2}_{\theta} M_n(\theta) -D^{2}_{\theta} M_n(\theta') \right\|
  \leq \omega(\delta) \times \frac{1}{n} \sum_{k=1}^{n}\left(\left|\Psi[S_{\theta^*}(t_k)+\e_k,t_k]+ V_k\right|+2\right).
\end{equation*}
It follows that on the set $(\overline{\theta}_n \in U )$
$$
    \|{D^{2}_{\theta} M_n[\theta^*+s(\overline{\theta}_n-\theta^*)]-D^{2}_{\theta} M_n(\theta^*)\|
   \leq \omega(\|\overline{\theta}_n-\theta^*}\|) \times \frac{1}{n} \sum_{k=1}^{n}\left(\left|\Psi[S_{\theta^*}(t_k)+\e _k,t_k]+ V_k\right|+2\right).
$$
By Lemma \ref{lem:1}, $\frac{1}{n} \sum_{k=1}^{n}\left|\Psi[S_{\theta^*}(t_k)+\e_k,t_k]+ V_k\right| = \oopst(1)$
so that, using the consistency of $\overline{\theta}_{n}$, Lemma \ref{lem:1} and Assumption \ref{ass:lsclt}:
\begin{equation*}
    \int_{0}^{1} D^{2}_{\theta} M_n[\theta^*+s(\overline{\theta}_n-\theta^*)] \, ds = 2 I_R(\theta^*) + \opst(1).
\end{equation*}
Finally, we obtain
\begin{multline*}
    \left(  I_R(\theta^*)+\opst(1) \right) \sqrt{n} (\overline{\theta}_n-\theta^*) =\\
    \frac{1}{\sqrt{n}} \sum_{k=1}^{n} \{ \Psi[S_{\theta^*}(t_k) +\e _k,t_k]+V_k-\Psi[S_{\theta^*}(t_k),t_k] \} \nabla_\theta \Psi[S_{\theta^*}(t_k),t_k] +\opst(1).
\end{multline*}
Using Assumption \ref{ass:lsclt},
the convergence in distribution to ${\cal N}\left(0,I_M^{-1}(\theta^*)\right)$ is a consequence of the Lindeberg-Feller Theorem
and Slutzky's Lemma.\\

Notice that, if $\hat{I}_M$ is a consistent estimator of $I_M(\theta^*)$,  by Slutsky's Lemma,\\
$\sqrt{n}{\hat{I}_M}^{1/2}(\overline{\theta}_n-\theta^* )$ converges in distribution to the centered standard gaussian distribution in $\R^{m}$,
which allows to build confidence regions with asymptotic known level. If the distribution of the trajectory noise $(\e_k )$
is known, one may use $\hat{I}_M=I_M(\overline{\theta}_{n})$. If the distribution of the noise is unknown, one could use bootstrap procedures to build confidence regions based on the empirical distribution of $\overline{\theta}_{n}$ using bootstrap replicates.


Another possibility occurs if one has a majoration
\begin{equation}
\label{eq:maj}
\esp \left\{ \Psi[S_{\theta^* }(t)+\e_{1},t] -\Psi[S_{\theta^* }(t),t]\right\}^2 \leq A^2,
\end{equation}
where $A$ denotes a known constant. Indeed, in such a case, $I_{\Psi}(\theta^*)$ is upper bounded
(in the natural ordering of positive symetric matrices) by
$A^{2}I_{R}(\theta^*)$, so that $I_{M}^{-1}(\theta^*)$ is upper bounded by
$(A^{2}+\sigma^2)I_{R}^{-1}(\theta^*)$, and one may use $(A^2+\sigma^2)I_R^{-1}(\overline{\theta}_n)$ as variance matrix to obtain conservative confidence regions.


\subsection{Application to BOT}
\label{subsec:ls:bot}
To apply the results to BOT, one has to see whether Assumptions  \ref{ass:bruit1}, \ref{ass:obs}, \ref{ass:esp}, \ref{ass:consi} and \ref{ass:lsclt} hold and
if $I_R(\theta^*)$ is non singular.\\
Assumption \ref{ass:obs} is the usual observability assumption which holds for
models such as uniform linear motion if the observer does not
move itself along uniform linear motion
, or a sequence of uniform linear and circular motions, if the observer does not move along uniform linear motion
or circular motion in the same time intervals as the target. Various observability properties are proved in
\cite{landelle}.\\
Assumptions \ref{ass:consi} and \ref{ass:lsclt} hold as soon as the trajectory model $S_{\theta}(t)$ is twice differentiable for all $t$ as a function of $\theta$ and
the denominator in (\ref{eq:psi:bot}) may not be $0$, that is
the bearing exact measurements of the non noisy
possible trajectory stay inside an interval with length $\pi$. This may be seen as an assumption on the
manoeuvres of the observer. This is a usual assumption in BOT literature.
The fact that $I_R(\theta^*)$ is non singular
is equivalent to the observability assumptions for linear models. Let us introduce such models.\\
Let $e_{1}(t),\ldots,e_{p}(t)$ be continuous functions on $[0,1]$, $\theta = (a_{1},\ldots,a_{p},b_{1},\ldots,b_{p})^T$,
\begin{equation}
\label{eq:linmod}
S_{\theta}(t)=\left(\begin{array}{l}
a_{1}e_{1}(t) + \ldots + a_{p}e_{p}(t)\\
b_{1}e_{1}(t) + \ldots + b_{p}e_{p}(t)
\end{array}
\right).
\end{equation}
Then
\begin{prop}
Under model (\ref{eq:linmod}), Assumption \ref{ass:obs} holds if and only if $I_R(\theta^*)$
is non singular.
\end{prop}
\begin{proof}\\
Let $\theta^{*} = (a_{1}^{*},\ldots,a_{p}^{*},b_{1}^{*},\ldots,b_{p}^{*})^T$.
Let
$$
m(\theta,t)=\frac{ S_\theta(t)_{2}-O_2(t)}{S_\theta(t)_{1}-O_1(t)}.
$$
Simple algebra gives that $\Psi [S_{\theta}(t),t]=\Psi [S_{\theta}^{*}(t),t]$ if and only if
$$
\sum_{k=1}^{p}(b_{k}-b_{k}^{*})e_{k}(t)-\sum_{k=1}^{p}(a_{k}-a_{k}^{*})e_{k}(t)m(\theta^{*},t)=0,
$$
so that Assumption \ref{ass:obs} holds if and only if the functions
$e_{1}(t),\ldots,e_{p}(t), e_{1}(t)m(\theta^{*},t),\ldots,e_{p}(t)m(\theta^{*},t)$ are linearly independent
in the space of continuous functions on $[0,1]$.\\
Also, for $i=1,\ldots,p$:
$$
\frac{\partial}{\partial a_{i}} \arctan m(\theta^{*},t) = - \left(\frac{1}{1+m(\theta^{*},t)^{2}}\right)\left(\frac{1}{S_\theta(t)_{1}-O_1(t)}\right)
e_{i}(t)m(\theta^{*},t)
$$
and
$$
\frac{\partial}{\partial b_{i}} \arctan m(\theta^{*},t) = \left(\frac{1}{1+m(\theta^{*},t)^{2}}\right)\left(\frac{1}{S_\theta(t)_{1}-O_1(t)}\right)
e_{i}(t),
$$
so that $I_R(\theta^*)$
is non singular if and only if the functions
$e_{1}(t),\ldots,e_{p}(t), e_{1}(t)m(\theta^{*},t),\ldots,e_{p}(t)m(\theta^{*},t)$ are linearly independent
in the space of continuous functions on $[0,1]$, which ends the proof.
\end{proof}

Thus under model (\ref{eq:linmod}), if the trajectory of the observer is such that $O_{2}(t)-\sum_{k=1}^{p}b_{k}^{*}e_{k}(t)\neq 0$ for all $t\in [0,1]$ and Assumption \ref{ass:obs} holds,
Assumptions \ref{ass:consi} and \ref{ass:lsclt} hold and
$I_R(\theta^*)$ is non singular.\\

What remains to be seen is whether Assumption \ref{ass:esp} holds, and it is the case under a simple assumption
on the distribution of the trajectory noise:
\begin{assumption}
\label{ass:bruit2}
$\e_{1}$ has an isotropic distribution in $\R^{2}$.
\end{assumption}
We introduce some prior knowledge on the trajectory and on the variance of the trajectory noise to be able
to obtain conservative confidence regions.
\begin{assumption}
\label{ass:dist}
The trajectory model $(t,\theta) \mapsto S_\theta(t)$ is such that for all $(\theta,t)\in \Theta\times [0,1]$,
$\|O(t)-S_\theta(t)\|\geq R_{\min}$, and a constant number $A^{2}$  such that
\begin{equation*}
\pi^{2}\left(1+\pi^{-2/3}\right)^{3}  \frac{\esp \|\e_1\|^2}{R_{\min}^{2}} \leq A^2 
\end{equation*}
is known.
\end{assumption}
This condition makes sense since in the context of passive tracking one usually assumes that the distance between target and observer is quite large. 

\begin{thm}
\label{theo:BOT:1}
If the trajectory model $(t,\theta)\mapsto S_{\theta}(t)$ and the move of the observer are such that
Assumptions \ref{ass:obs},  \ref{ass:consi}, \ref{ass:lsclt} and \ref{ass:dist} hold and $I_R(\theta^*)$ is non singular,\\
or if the trajectory model is (\ref{eq:linmod}), the trajectory of the observer is such that $O_{2}(t)-\sum_{k=1}^{p}b_{k}^{*}e_{k}(t)\neq 0$ for all $t\in [0,1]$ and Assumption \ref{ass:obs} holds,\\
if moreover Assumption \ref{ass:bruit1} and \ref{ass:bruit2} hold,\\
then for any $\alpha >0$, if $C_{\alpha}$ is a region with coverage $1-\alpha$ for the standard gaussian
distribution in $\R^{m}$, then
$$
\liminf_{n\rightarrow +\infty} \PPst\left(\frac{\sqrt{n}}
{\sqrt{A^2+\sigma^{2}}}{I_{R}}^{1/2}(\overline{\theta}_{n})\left(\overline{\theta}_{n}-\theta^*\right) \in C_{\alpha}\right) \geq 1-\alpha.
$$
\end{thm}
\begin{proof}
\\
Under Assumption \ref{ass:bruit2}, let the density of  $\varepsilon_1$ be $F(\|\e\|)$. Recall that the trajectory of the observer is
$(O(t) )_{t \in [0,1]}$. Let $\beta (t)=\arctan [S_\theta(t)_{2}-O_2(t)]/[S_\theta(t)_{1}-O_1(t)]=\Psi[S_\theta(t),t]$.
\begin{multline*}
     \esp \{\Psi[S_\theta(t)+\e_{1},t]\} = \int\!\!\! \int_{\R \times (-\pi,\pi)}  \arctan\left[
   \frac{ S_\theta(t)_{2}-O_2(t)+r \sin \alpha}{S_\theta(t)_{1}-O_1(t)+r\cos \alpha}
     \right] F(r)\, rdr d\alpha \, , \\
      = \beta (t) +\int\!\!\! \int_{\R \times (-\pi,\pi)}  \arctan \left( \frac{ r \sin( \alpha- \beta(t))}{ \|O(t)-S_\theta(t)\| +r\cos(\alpha- \beta(t))} \right) F(r)\, rdr d\alpha.
\end{multline*}

Let
\begin{equation*}
    G_{\theta,t}(r,\alpha) = \arctan \left( \frac{ r \sin \alpha}{ \|O(t)-S_\theta(t)\| +r\cos \alpha} \right).
\end{equation*}
Then,
\begin{equation*}
    \esp \{\Psi[S_\theta(t)+\e_{1},t]\} = \Psi[S_\theta(t),t] +\int\!\!\! \int_{\R \times (-\pi,\pi)} G_{\theta,t}(r,\alpha) F(r)\, rdr d\alpha .
\end{equation*}
But for any $r>0$,  for any $\alpha$, $G_{\theta,t}(r,-\alpha)=G_{\theta,t}(r,\alpha)$ so that
\begin{equation*}
    \esp \{\Psi[S_\theta(t)+\e_{1},t]\} = \Psi[S_\theta(t),t].
\end{equation*}
Now,
$$
\Psi[S_{\theta^* }(t)+\e_{1},t] -\Psi[S_{\theta^* }(t),t]=\int_{0}^{1}\nabla_x \Psi[S_{\theta^* }(t)+h\e_{1},t]^{T}
\e_1 dh,
$$
and direct calculations provide
\begin{equation*}
    \left\| \nabla_x \Psi[x,t] \right\| =  \|O(t)-x\|^{-1}.
\end{equation*}
Thus for any $a\in ]0,1[$:
\begin{eqnarray*}
\esp \left\{ \Psi[S_{\theta^* }(t)+\e_{1},t] -\Psi[S_{\theta^* }(t),t]\right\}^2 &\leq &
\pi^{2}\PP\left(\|\e_{1}\|\geq a \|O(t)-\Psi[S_{\theta^* }(t)\|\right) \\
&&+ \frac{\esp \|\e_1\|^2}{(1-a)^{2}\|O(t)-\Psi[S_{\theta^* }(t)\|^{2}}\\
&\leq &
\pi^{2}\PP\left(\|\e_{1}\|\geq a R_{\min}\right) + \frac{\esp \|\e_1\|^2}{(1-a)^{2}R_{\min}^{2}}
\end{eqnarray*}
since $|\Psi(u)-\Psi (v)|\leq \pi$ for any real numbers $u$ and $v$, and by using the triangular inequality
and Assumption \ref{ass:dist}. \\
But Tchebychev inequality leads to
\begin{equation}
\label{eq:conservative_majoration}
\esp \left\{ \Psi[S_{\theta^* }(t)+\e_{1},t] -\Psi[S_{\theta^* }(t),t]\right\}^2
\leq \frac{\esp \|\e_1\|^2}{R_{\min}^{2}}\left(\frac{\pi^{2}}{a^{2}}+\frac{1}{(1-a)^{2}}\right)
\end{equation}
which is minimum for $a=\frac{1}{1+\pi^{-2/3}}$ leading to $\left(\frac{\pi^{2}}{a^{2}}+\frac{1}{(1-a)^{2}}\right)=\pi^{2}\left(1+\pi^{-2/3}\right)^{3}$ and
$$
\esp \left\{ \Psi[S_{\theta^* }(t)+\e_{1},t] -\Psi[S_{\theta^* }(t),t]\right\}^2
\leq A^{2}.
$$
To conclude one may apply the concluding remark of Section \ref{subsec:ls:as} to obtain asymptotic conservative confidence regions for $\theta$.\\
\end{proof}

\section{Likelihood and efficiency}
\label{sec:lik}
Let $\cal F$ be the set of probability densities $f$ on $\R^{d}$ such that
for all $t\in [0,1]$, for all $\theta \in \Theta$,
\begin{equation}
\label{eq:dens}
\int_{\R^{d}} \Psi[S_\theta(t)+\e ,t] f \left( \e \right) d\e =  \Psi[S_\theta(t),t].
\end{equation}
We will replace Assumptions \ref{ass:bruit1} and \ref{ass:esp} by
\begin{assumption}
\label{ass:bruit3}
$(\varepsilon_k )_{k\in\N}$ is a sequence of i.i.d. random variables with density $f^{*}\in {\cal F}$.
\end{assumption}
The normalized log-likelihood is the function on $\Theta \times {\cal F}$
\begin{equation}
\label{eq:LogLikelihood}
    J_n(\theta,f) = \frac{1}{n} \sum_{k=1}^{n} \log \left( \int g \left\{ Y_k-\Psi[S_\theta(t_k)+u,t_k] \right\} \, f(u)du \right).
\end{equation}
Define
$$
G\left((\varepsilon, V), t; \theta\right)=  \log \left(  \int g\left\{ \Psi[S_{\theta^*}(t)+\e,t]+V- \Psi[S_\theta(t)+u,t] \right\}\,f(u)\,du \right),
$$
where $(\epsilon,V)$ has the same distribution as $(\epsilon_{1},V_{1})$.\\
As soon as for any $(\theta,f)\in \Theta \times {\cal F}$, it is possible to apply
Lemma \ref{lem:1} to
$G\left((\cdot), \cdot;\theta\right)$,
$J_n(\theta,f)$ converges in probability to
\begin{equation}
\label{eq:liklim}
    J(\theta , f) =  \int_{0}^{1} \int_{\R^{d}} \int_{\R}
    \log \left(  \int g\left\{ \Psi[S_{\theta^*}(t)+\e,t]+v- \Psi[S_\theta(t)+u,t] \right\}\,f(u)\,du \right) \,g(v)f^{*}(\e) \, dv\, d\e\, dt.
\end{equation}
Let
$$p_{(\theta,f)}\left(z, t\right)=\int g\left\{ z- \Psi[S_\theta(t)+u,t] \right\}\,f(u)\,du
$$
be the density, for fixed $t$, of the random variable $Z=\Psi[S_\theta(t)+U,t]+V$ where $U$ is a random
variable in $\R^{d}$ with density $f$ independent of the random variable $V$ in $\R$ with density $g$.
Thus, $p_{(\theta^{*},f^{*})}\left(\cdot , t_k\right)$ is the probability density of $Y_k$.
Then, the change of variable $z=\Psi[S_{\theta^*}(t)+\e,t]+v$ in
$ \int_{\R}
    \log \left(  \int g\left\{ \Psi[S_{\theta^*}(t)+\e,t]+v- \Psi[S_\theta(t)+u,t] \right\}\,f(u)\,du \right) \,g(v) \, dv$ leads to
$$
 J(\theta , f) = \int\left[\int p_{(\theta^{*},f^{*})}\left(z, t\right) \log p_{(\theta,f)}\left(z, t\right) dz
 \right] dt.
$$
Thus, for any $(\theta,f)\in \Theta \times {\cal F}$,
$$
J(\theta^* , f^*)\geq J(\theta , f),
$$
and $J(\theta^* , f^*)= J(\theta , f)$ if and only if $t$ a.e. $p_{(\theta,f)}\left(z, t\right)=p_{(\theta^{*},f^{*})}\left(z, t\right)$ $z$ a.e., that is the probability distribution of
$\Psi[S_\theta(t)+U,t]+V$,where $U$ is a random
variable in $\R^{d}$ with density $f$ independent of the random variable $V$ in $\R$ with density $g$, is the
same as that of $\Psi[S_{\theta*}(t)+U^{*},t]+V$,where $U^*$ is a random
variable in $\R^{d}$ with density $f^*$ independent of the random variable $V$ in $\R$ with density $g$.
But if $f\in{\cal F}$ and $f^{*}\in{\cal F}$, taking expectations
leads to the fact that, $t$ a.e., $\Psi[S_\theta(t),t]=\Psi[S_{\theta*}(t),t]$, so that $\theta = \theta^*$ if
Assumption \ref{ass:obs} holds. In other words, $J(\theta , f)$ is maximum only for $\theta = \theta^*$.\\
Following the same lines as for the LSE, we may thus easily obtain that, if the probability  density $f^*$ is known, the
parametric maximum likelihood
estimator is consistent and asymptotically gaussian.
Define the parametric maximum likelihood
estimator as :
$$
\tilde{\theta}_n = \arg \max_{\theta \in \Theta} J_n (\theta, f^*).
$$
where $\arg \max_{\theta \in \Theta} J_n(\theta, f^*)$ is any maximizer of $J_{n}(\cdot, f^*)$.\\
If for any $\theta\in\Theta$, there exists a small open ball containing $\theta$ such that
Lemma \ref{lem:1} applies to
$\sup_{\theta\in U}G\left((\cdot), \cdot;\theta\right)$, it is possible, as in \cite{vandervaart:1998}
Theorem 5.14,
to strengthen the convergence of $J_{n}(\theta, f^*)$ to $J(\theta, f^*)$ in a uniforme one. The consistency of
$\tilde{\theta}_n$ follows:
\begin{thm}
\label{consipara}
Under assumptions \ref{ass:obs} and \ref{ass:bruit3}, if moreover
Lemma \ref{lem:1} applies to
$\sup_{\theta\in U}G\left((\cdot), \cdot;\theta\right)$, then the estimator $\tilde{\theta}_n$
is consistent.
\end{thm}
We will use the notation $Y(t)$ for $Y(t)=\Psi[S_{\theta^*}(t)+\e_{1},t]+V_{1}$ to simplify the writing of
some integrals.
We shall introduce the assumptions we need to prove the asymptotic distribution of $\tilde{\theta}_n$:
\begin{assumption}
\label{ass:paraclt}
For all $(z,t)\in \R \times [0,1]$, the function $\theta \mapsto p_{(\theta,f^*)}\left(z, t\right)$ is twice
continuously differentiable.\\
For any $\theta \in \Theta$, $t\mapsto \esp \|\nabla_{\theta} \log p_{(\theta,f^*)}\left(Y(t), t\right)\|^{2}$ is finite and continuous.\\
There exists a neighborhood $U$ of $\theta^*$ such that for all $\theta\in U$,
$t\mapsto \esp  D^{2}_{\theta} \log p_{(\theta,f^*)}\left(Y(t), t\right)$ is finite and continuous.\\
Lemma \ref{lem:1} applies to $\log p_{(\theta,f)}\left(Y(t), t\right)$, for all $\theta$,
to $\|\nabla_{\theta} \log p_{(\theta,f^*)}\left(Y(t), t\right)\|^{2}$ and all components of
$D^{2}_{\theta} \log p_{(\theta,f^*)}\left(Y(t), t\right)$ for $\theta\in U$.
\end{assumption}
Introduce the parametric Fisher information matrix:
\begin{equation*}
I(\theta) =\int_{0}^{1} \esp \left[\frac{\nabla_\theta p_{(\theta,f^*)}}{p_{(\theta,f^*)}}\left(Y(t), t\right)
\frac{\nabla_\theta p_{(\theta,f^*)}}{p_{(\theta,f^*)}}\left(Y(t),t\right)\right] dt
\end{equation*}
\begin{thm}
\label{theo:lrtpara}
Under assumptions \ref{ass:obs}, \ref{ass:bruit3} and \ref{ass:paraclt},
$\tilde{\theta}_n$ converges in probability to $\theta^{*}$ as $n$ tends to infinity.\\
Moreover, if $I(\theta^*)$ is non singular,
\begin{equation*}
\sqrt{n}(\tilde{\theta}_n-\theta^*) = I^{-1}(\theta^*)\frac{1}{\sqrt{n}}\sum_{k=1}^{n}\frac{ \nabla_\theta p_{(\theta^*,f^*)}}{p_{(\theta^*,f^*)}}\left(Y_k,t_k\right) + \opst(1),
\end{equation*}
and
$\sqrt{n}(\tilde{\theta}_n-\theta^*)$ converges in distribution as $n$ tends to infinity to ${\cal N}(0,I^{-1}(\theta^*))$.
\end{thm}
The proof follows the same lines as that of Theorems \ref{theo:ls:cons} and \ref{theo:ls:clt} and is left to the reader.\\
Notice that under the same assumptions, it is easy to prove that the parametric model is locally asymptotically normal in the sense
of Le Cam (see \cite{lecam:1986}) so that if $I(\theta^*)$ is singular, there exist no regular estimator of $\theta$ which is
$\sqrt{n}$-consistent. Thus if $I_{R}(\theta^*)$ is non singular and the assumptions in Theorem \ref{theo:ls:clt}
hold,
in which case
the LSE is regular $\sqrt{n}$-consistent, then $I(\theta^*)$ is also non singular.\\

To investigate the optimality of possible estimators in the semiparametric situation, with $f^*$ unknown but known to belong to $\cal F$, we use Le Cam's theory as developed for non i.i.d. observations by Mc Neney and Wellner \cite{mcneney:wellner:2000}.
Introduce the set $\cal B$ of integrable functions $b$ on $\R^d$ such that:
\begin{itemize}
\item $\int b(u)du =0$ and $\exists \delta >0,\; f^{*}+\delta b \geq 0$,
\item
for all $t\in [0,1]$, for all $\theta \in \Theta$,
$$
\int_{\R^{d}} \Psi \left[ S_{\theta}(t)+\e,t\right] b(\e)d\e =0.
$$
\item
$$
\int_{0}^{1}\esp \left(\frac{\int g(Y(t)-\Psi [ S_{\theta^{*}}(t)+u,t ])b(u)du}
{\int g(Y(t)-\Psi [ S_{\theta^{*}}(t)+u,t ])f^{*}(u)du}\right)^{2}   dt < \infty .
$$
\end{itemize}
Let ${\cal H}=\R^{m} \times {\cal B}$ be endowed with the inner product
\begin{eqnarray*}
\langle (a_{1},b_{1}),(a_{2},b_{2})\rangle_{{\cal H}} &=\int_{0}^{1}\esp &\left\{\left(\frac{ \nabla_\theta p_{(\theta^*,f^*)}^{T}}{p_{(\theta^*,f^*)}}\left(Y(t),t \right)\cdot a_{1} + \frac{\int g(Y(t)-\Psi [ S_{\theta^{*}}(t)+u,t ])b_{1}(u)du}
{\int g(Y(t)-\Psi [ S_{\theta^{*}}(t)+u,t ])f^{*}(u)du}\right)\right.\\
&&\left.\left(\frac{ \nabla_\theta p_{(\theta^*,f^*)}^{T}}{p_{(\theta^*,f^*)}}\left(Y(t),t \right)\cdot a_{2} + \frac{\int g(Y(t)-\Psi [ S_{\theta^{*}}(t)+u,t ])b_{2}(u)du}
{\int g(Y(t)-\Psi [ S_{\theta^{*}}(t)+u,t ])f^{*}(u)du}\right)
\right\}dt.
\end{eqnarray*}
We will need only local smoothness, so we introduce:
\begin{assumption}
\label{ass:lan}
There exists a neighborhood $U$ of $\theta^*$ such that for $\theta \in U$:\\
For all $(z,t)\in \R \times [0,1]$, the function $\theta \mapsto p_{(\theta,f^*)}\left(z, t\right)$ is twice
continuously differentiable.\\
$t\mapsto \esp \|\nabla_{\theta} \log p_{(\theta,f^*)}\left(Y(t), t\right)\|^{2}$ is finite and continuous.\\
$t\mapsto \esp  D^{2}_{\theta} \log p_{(\theta,f^*)}\left(Y(t), t\right)$ is finite and continuous.\\
For any $b\in {\cal B}$, for all $(z,t)\in \R \times [0,1]$, $\theta \mapsto \int g(z-\Psi \left[ S_{\theta}(t)+u,t\right])b(u)du$
is continuously differentiable and
$t\mapsto \esp\left\|\frac{\nabla_{\theta}\int g(Y(t)-\Psi \left[ S_{\theta}(t)+u,t\right])b(u)du}{p_{(\theta^*,f^*)}(Y(t),t)}\right\|$ is finite and continuous.\\
Lemma \ref{lem:1} applies to  $\|\nabla_{\theta} \log p_{(\theta,f^*)}\left(Y(t), t\right)\|^{2}$,all components of
$D^{2}_{\theta} \log p_{(\theta,f^*)}\left(Y(t), t\right)$ and
$\left\|\frac{\nabla_{\theta}\int g(Y(t)-\Psi \left[ S_{\theta}(t)+u,t\right])b(u)du}{p_{(\theta^*,f^*)}(Y(t),t)}\right\|$
for $\theta\in U$.
\end{assumption}
Let $\PP_{n,(\theta,f)}$ be the distribution of $Y_{1},\ldots,Y_{n}$ when the parameter is $\theta$ and the density of
the trajectory noise is $f$. For $(\theta,f)\in \Theta \times {\cal F}$, let
$$\Lambda_{n} \left(\theta , f\right)=\log \frac{d\PP_{n,(\theta,f)}(Y_{1},\ldots,Y_{n})}{d\PP_{n,(\theta^*,f^*)}(Y_{1},\ldots,Y_{n})}
= J_{n}\left(\theta , f\right)-J_{n}\left(\theta^{*} , f^{*} \right).$$
Then
\begin{prop}
\label{prop:lan}
Assume that Assumption \ref{ass:lan}  holds. Then the sequence of statistical models $(\PP_{n,(\theta,f)})_{\theta\in\Theta, f\in{\cal F}}$ is locally asymptotically
normal with tangent space ${\cal H}$, that is, for $(a,b)\in{\cal H}$,
$$
\Lambda_{n}\left(\theta^* + \frac{a}{\sqrt{n}}, f^* + \frac{b}{\sqrt{n}}\right)
= W_{n}\left(a,b\right) - \frac{1}{2}\|\left(a,b\right)\|^{2}_{\cal H}+ \opst(1),
$$
where
$$
W_{n}\left(a,b\right)=\frac{1}{\sqrt{n}}\sum_{k=1}^{n} \left(\frac{ \nabla_\theta p_{(\theta^*,f^*)}^{T}}{p_{(\theta^*,f^*)}}\left(Y_k,t_k \right)\cdot a  + \frac{\int g(Y_{k}-\Psi [ S_{\theta^{*}}(t_k)+u,t_k ])b (u)du}
{\int g(Y_{k}-\Psi [ S_{\theta^{*}}(t_k)+u,t_k ])f^{*}(u)du}\right)
$$
and for any finite subset $h_1,\ldots,h_q \in {\cal H}$, the random vector $(W_n(h_1),\ldots,W_n(h_q))$ converges in distribution to the centered Gaussian vector with covariance $\langle h_i,h_j \rangle_{{\cal H}}$.
\end{prop}
\begin{proof}
\begin{multline*}
\Lambda_{n}\left(\theta^* + \frac{a}{\sqrt{n}}, f^* + \frac{b}{\sqrt{n}}\right)\\
=\sum_{k=1}^{n} \log \left(1+\frac{p_{(\theta^* + \frac{a}{\sqrt{n}},f^*)}-p_{(\theta^*,f^*)}}{p_{(\theta^*,f^*)}}\left(Y_k,t_k \right) +
\frac{1}{\sqrt{n}} \frac{\int g(Y_k -\Psi \left[ S_{\theta^* + \frac{a}{\sqrt{n}}}(t_k)+u,t_k\right])b(u)du}{p_{(\theta^*,f^*)}(Y_k,t_k)}\right)\\
=
W_{n}\left(a,b\right) - \frac{1}{2}\|\left(a,b\right)\|^{2}_{\cal H}+ \opst(1),
\end{multline*}
by using: Taylor expansion till second order of $\log (1+u)$,  Taylor expansion till second order of
$\theta \mapsto p_{(\theta,f^*)}\left(z, t\right)$ and Taylor expansion till first order of
$\theta \mapsto \int g(z-\Psi \left[ S_{\theta}(t)+u,t\right])b(u)du$, which gives the
first order term $W_{n}\left(a,b\right)$, and then applying Lemma \ref{lem:1} to the second order terms
to get  $\frac{1}{2}\|\left(a,b\right)\|^{2}_{\cal H}+ \opst(1)$.\\
The convergence of $(W_{n}(h))_{h\in{\cal H}}$ to the isonormal process on $\cal H$ comes from Lindeberg Theorem
applied to finite dimensional marginals.
\end{proof}

The interest of Proposition \ref{prop:lan} is that it gives indications on the limitations on the estimation of $\theta^*$ when $f^*$ is unknown. Indeed, the efficient Fisher information $I^*$ is given by:
$$
\inf_{b\in{\cal B}} \|\left(a,b\right)\|_{\cal H}^{2}=a^{T} I^* a,
$$
and if $I^*$ is non singular,
any regular estimator $\widehat{\theta}$ that converges at speed $\sqrt{n}$ has asymptotic covariance $\Sigma$
which is lower bounded (in the sense of positive definite matrices) by $(I^*)^{-1}$.\\
In case $I_{R}(\theta^*)$ is non singular and the assumptions in Theorem \ref{theo:ls:clt} hold, one may deduce that
$I^*$ is non singular.

\subsection{Application to BOT}

As seen in Section \ref{subsec:ls:bot}, the set of isotropic densities is a subset of $\cal F$.
If $g$ is twice differentiable, positive and upper bounded, if the trajectory model
$\theta\mapsto S_{\theta}(t)$ is twice differentiable for all $t\in[0,1]$, then Assumptions
\ref{ass:paraclt} and \ref{ass:lan} hold under almost any trajectory of the observer.
Indeed, one may apply Lebesgue's Theorem to obtain derivatives
of integrals, and use the fact that the function $z\mapsto \arctan z$ is infinitely differentiable,  has vanishing derivatives at infinity, is bounded and has two bounded derivatives, so that if the trajectory of the observer is such that, for all $\theta$, the
set of times $t$ and points $u$ such that $\Psi(S_{\theta}(t)+u,t)$ is $-\frac{\pi}{2}$ or $\frac{\pi}{2}$ is negligible, then
the smoothness assumptions hold.\\
Moreover, as seen again in Section \ref{subsec:ls:bot}, if the trajectory model is  (\ref{eq:linmod})
and satisfies Assumption  \ref{ass:obs},
then $I_{R}(\theta^*)$ is non singular, so that the efficient Fisher information $I^*$ is non singular, and
all results of Section \ref{sec:lik} apply.

\section{Further considerations}
\label{sec:dep}

It would be of great interest to have a more explicit general expression of $I^*$, and of greater interest to
exhibit an asymptotically regular and efficient estimator $\widehat{\theta}$. If one could approximate the
profile likelihood $\sup_{f\in{\cal F}} J_{n}(\theta,f)$, one could hope that the maximizer $\widehat{\theta}$
of it be a good candidate.\\
Another possibility would be to use Bayesian estimators. Indeed, in the parametric context, the Bernstein-von Mises
Theorem tells us that asymptotically, the posterior distribution of the parameter is  gaussian, centered
at the maximum likelihood estimator, and with variance the inverse of Fisher information (see \cite{vandervaart:1998} for a nice
presentation). Extensions to semiparametric situations are now available, see \cite{castillo:2008}. To obtain semiparametric
Bernstein-von Mises Theorems, one has to verify assumptions relating the particular model and the choice of the
non parametric prior. This could be the object of further work. Then, with an adequate choice of the prior on
$\Theta\times {\cal F}$, taking advantage of  MCMC computations, one could propose bayesian methods to estimate
$\theta^*$ (mean posterior, maximum posterior, median posterior for example).\\

To extend the results of the preceding sections in the case where the trajectory noise
is no longer a sequence of i.i.d. random variables, one needs to prove laws of large numbers
and central limit theorems for empirical sums such as
$\frac{1}{n}\sum_{k=1}^{n} F\left( \e_{k},t_{k}\right)$, we prove some below for
stationary weakly dependent sequences $(\e_{k})_{k\in\N}$. In such a case, if $M(\theta)$ and $J(\theta, f^*)$
are still the limits of $M_{n}(\theta)$ and $J_n(\theta, f^*)$ respectively, then asymptotics for $\overline{\theta}_n$
and $\tilde{\theta}_n$ could be obtained. Here, $J_n(\theta, f^*)$ is no longer the normalized log-likelihood, rather
the marginal normalized log-likelihood, but $J(\theta, f^*)$ is still a contrast function.\\
Since the convergence of the expectation relies on purely deterministic arguments (Rieman integrability), we focus on
centered functions. We assume in this section that
\begin{assumption}
\label{ass:bruitdep}
$(\e_{k})_{k\in\N}$ is a stationary sequence of
random variables such that for all $t\in [0,1]$
$$
\esp \left[F\left( \e_{1},t\right)\right]=0.
$$
\end{assumption}
Denote by $(\alpha_{k})_{k\in\N}$ the strong mixing coefficients of the sequence $(\e_{k})_{k\in\N}$
defined as in \cite{rio:2000}, that is, for $k\geq 1$,
$$
\alpha_{k} = 2\sup_{\ell\in\N,A \in \sigma(\e_{i}:i\leq \ell),B \in \sigma(\e_{i}:i\geq k+\ell) }
\left\vert\PP\left(A \cap B\right)-\PP\left(A \right)\PP\left(B\right)\right\vert.
$$
and $\alpha_{0}=\frac{1}{2}$.
Notice that they are also an upper bound for the strong mixing coefficients of the sequence $(F\left( \e_{k},t_{k}\right))_{k\in\N}$
for any sequence $(t_{k})_{k\in\N}$ of real numbers in $[0,1]$.
\begin{prop}
\label{prop:1}
Under Assumption \ref{ass:bruitdep}, if $\alpha_{k}$ tends to $0$ as $k\rightarrow +\infty$, if
$\sup_{t\in [0,1]}\esp |F(\e_{1},t)|$ is finite and
$\lim_{M\rightarrow +\infty}\sup_{t\in [0,1]}\esp \{|F(\e_{1},t)|1_{ |F(\e_{1},t)|>M}\}=0$,
then
$$\frac{1}{n}\sum_{k=1}^{n} F\left( \e_{k},t_{k}\right)
$$
converges in probability to $0$ as $n$ tends to infinity.
\end{prop}
\begin{proof}
\\
Using Ibragimov's inequality (\cite{Ibragimov:1962}), for any $M$:
\begin{eqnarray*}
\Var\left(\frac{1}{n}\sum_{k=1}^{n} F\left( \e_{k},t_{k}\right)1_{ |F(\e_{k},t_{k})|\leq M}\right)
&=&\frac{1}{n^{2}}\sum_{i=1}^{n}\sum_{j=1}^{n}\Cov\left(F\left( \e_{i},t_{i}\right)1_{ |F(\e_{i},t_{i})|\leq M}
;F\left( \e_{j},t_{i}\right)1_{ |F(\e_{i},t_{i})|\leq M}\right)
\\
&\leq &\frac{2 M^{2}}{n^{2}}\sum_{i=1}^{n}\sum_{j=1}^{n}\alpha_{|i-j|}\\
&\leq & \frac{2 M^{2}}{n }\sum_{k=0}^{n-1}\alpha_{k}
\end{eqnarray*}
which tends to $0$ by Cesaro as $n\rightarrow +\infty$.\\
The end of the proof is similar to that of Lemma \ref{lem:1}.
\end{proof}

Define now
$$
\alpha^{-1}\left(u\right)=\inf\left\{k\in\N;:\;\alpha_{k}\leq u\right\}
=\sum_{i\geq 0}1_{u<\alpha_{i}}.
$$
Define also for any $t\in [0,1]$,
$$
Q_{t}\left(u\right)=\inf\left\{x\in\R;:\;\PP\left(\vert F(\e_{1},t)\vert>x\right)\leq u\right\},
$$
and
$$
Q\left(u\right)=\sup_{t\in [0,1]}Q_{t}\left(u\right).
$$
We shall assume that
\begin{assumption}
\label{ass:depgene}
$$
\int_{0}^{1} \alpha^{-1}\left(u\right)Q^{2}\left(u\right)du < +\infty,
$$
\end{assumption}
which is the same as the convergence of the series
$$
\sum_{k\geq 0}\int_{0}^{\alpha_{k}}Q^{2}\left(u\right)du .
$$
Applying Theorem 1.1 in \cite{rio:2000} one gets for any $t\in[0,1]$ and $k\geq 0$:
$$
\left\vert\Cov\left(F\left( \e_{0},t\right);F\left( \e_{k},t\right)\right)\right\vert
\leq 2 \int_{0}^{\alpha_{k}}Q^{2}\left(u\right)du,
$$
so that if Assumption \ref{ass:depgene} holds, one may define
\begin{equation}
\label{sigma2}
\gamma^{2}=\int_{0}^{1}\Var F\left( \e_{0},t\right)dt + 2 \sum_{k=1}^{+\infty}
\int_{0}^{1}\Cov\left(F\left( \e_{0},t\right);F\left( \e_{k},t\right)\right)dt.
\end{equation}
Now:
\begin{prop}
\label{prop:2}
Under Assumptions \ref{ass:bruitdep} and \ref{ass:depgene}, if $\sigma^{2} >0$ and if
for any integer $k$, the real function
$(t,u)\rightarrow \Cov\left(F\left( \e_{0},t\right);F\left( \e_{k},u\right)\right)$
is continuous on $[0,1]^{2}$, then
$$
\frac{1}{\sqrt{n}}\sum_{k=1}^{n} F\left( \e_{k},t_{k}\right)
$$
converges in distribution to ${\cal N}(0,\gamma^{2})$ as $n$ tends to infinity.
\end{prop}
\begin{proof}
\\
Let
$
S_{n}=\sum_{k=1}^{n} F\left( \e_{k},t_{k}\right).
$
First of all, let us prove that $\frac{\Var S_{n}}{n}$ converges to $\sigma^{2}$
as $n$ tends to infinity.
\begin{eqnarray*}
\frac{\Var S_{n}}{n}&=&\frac{1}{n}\sum_{i=1}^{n}\sum_{j=1}^{n}
\Cov\left(F\left( \e_{i},t_{i}\right);F\left( \e_{j},t_{j}\right)\right)\\
&=&\frac{1}{n}\sum_{k=1-n}^{n-1}\sum_{i=1\vee (1-k)}^{n\wedge (n-k)}
\Cov\left(F\left( \e_{0},t_{i}\right);F\left( \e_{k},t_{i+k}\right)\right).
\end{eqnarray*}
For any $K\geq 1$, using again Theorem 1.1 in \cite{rio:2000}
$$
\left\vert\frac{1}{n}\sum_{K\leq |k|\leq n-1}\sum_{i=1\vee (1-k)}^{n\wedge (n-k)}
\Cov\left(F\left( \e_{0},t_{i}\right);F\left( \e_{k},t_{i+k}\right)\right)\right\vert
\leq 2 \sum_{k\geq K}\int_{0}^{\alpha_{k}}Q^{2}\left(u\right)du
$$
which is smaller than any positive $\epsilon$ for big enough $K$ under Assumption \ref{ass:depgene}.\\
Now, for any fixed integer $k$,
\begin{eqnarray*}
&&\left\vert\frac{1}{n} \sum_{i=1\vee (1-k)}^{n\wedge (n-k)}
\Cov\left(F\left( \e_{0},t_{i}\right);F\left( \e_{k},t_{i+k}\right)\right)
-\int_{0}^{1}\Cov\left(F\left( \e_{0},t\right);F\left( \e_{k},t\right)\right)dt\right\vert\\
&\leq &
\sup_{t,u\in[0,1],|t-u|\leq \frac{k}{n}}\left\vert \Cov\left(F\left( \e_{0},t\right);F\left( \e_{k},t+u\right)\right)-\Cov\left(F\left( \e_{0},t\right);F\left( \e_{k},t\right)\right)\right\vert\\
&&+ \frac{k}{n}\sup_{t\in[0,1]}\left\vert \Cov\left(F\left( \e_{0},t\right);F\left( \e_{k},t\right)\right)\right\vert
\end{eqnarray*}
which goes to $0$ as $n$ tends to infinity under the continuity assumption. The convergence of $\frac{\Var S_{n}}{n}$
to $\sigma^{2}$ follows.\\
The end of the proof of Proposition \ref{prop:2} is a direct application of Corollary 1 in \cite{rio:1995}.
\end{proof}

\section{Simulations}
\label{sec:simu}

The simulations have been realized using \texttt{Matlab}. The minimisation is made with the function \texttt{searchmin} by setting to  $2000$ the options \texttt{MaxFunEvals} and \texttt{MaxIter}, so that the method reaches the minimum.

For all the simulations, the observation time is of $20 \s$. The trajectory of the observer has a speed  with constant norm $\displaystyle \left\| \frac{\dd O(t)}{\dd t}\right\|$equal to $0.25 \kms$ and makes maneuvers with norm of acceleration $\displaystyle \left \| \frac{\dd^2 O(t)}{\dd t^2} \right\|$ of approximatively $50 \mss$. The trajectory is mainly composed of uniform linear motions and circular uniform motions. The different sequences of the trajectory of the platform are described in the following table. The null values of acceleration correspond to uniform linear motions and the others to uniform circular motion.

\begin{equation*}
    \begin{array}{l|c|c|c|c}
       \textrm{time interval (s)} & 0-6 & 7-10 & 11-14 & 15-20 \\
       \hline
       \textrm{norm of acceleration} \mathrm{(m/s^2)} & 50 & 0 & -55 & 0
    \end{array}
\end{equation*}

The positive and negative values for norm of acceleration correspond respectively to anticlockwise and clockwise circular motion. The transition sequences between circular motion and linear motion which are the time intervals  $[6,7]$, $[10,11]$, and $[14,15]$ are such that the whole trajectory is ${\cal C}^\infty$.

The assumed parametric model is a uniform linear motion with a speed of $0.27 \kms$. The parameter $\theta$ is defined by
\begin{equation*}
    \theta = \begin{pmatrix} x_0& y_0& v_x& v_y \end{pmatrix}^T \;.
\end{equation*}
where $(x_0,y_0)$ denotes the initial position and $(v_x,v_y)$ the speed vector. The parametric trajectory is then defined by
\begin{equation*}
    X_\theta(t) = \begin{pmatrix} x_0 + v_x t \\ y_0 + v_y t \end{pmatrix} \;.
\end{equation*}
The observation noise is a sequence of i.i.d centered Gaussian variables with variance $\sigma = 10^{-3} \rad$. The platform receives $2000$ observations.

For the first simulation, we consider a  sequence $(\varepsilon_k )_{k\in\N}$ of i.i.d Gaussian centered random variables with variance $\sigma_X^2 \times I_2$ and $\sigma_X = 10 \m$. The figure \ref{fig:BLSE_isotropic_iid_siml} shows the trajectory of the platform with a realization of a trajectory of the target and the parametric trajectory with parameter $\overline{\theta}_n$ and also the confidence area with level of $95 \%$ for the position at final time. The figure \ref{fig:MLE_isotropic_iid_siml} presents the same for the maximum likelihood estimator (MLE) $\tilde{\theta}_n$.

By using Monte-Carlo methods with $1000$ experiments, histograms of the coordinates of $\sqrt{n}(\overline{\theta}_n-\theta^*)$ are presented on figure \ref{fig:BLSE_isotropic_iid_histo} with the marginal probability densities of the asymptotic law ${\cal N}(0,I_M^{-1}(\theta^*))$ in dotted line. The empirical cumulative distribution functions of the coordinates of $\sqrt{n}(\overline{\theta}_n-\theta^*)$ are presented on figure \ref{fig:BLSE_isotropic_iid_cdf} juxtaposed to the marginal cumulative distributions of law ${\cal N}(0,I^{-1}_M(\theta^*))$. These two figures illustrate the convergence in distribution given by Theorem \ref{theo:ls:clt}, since the sequence $(\varepsilon_k)_{k \in \N}$ is an i.i.d. sequence of isotropic random variables.

\begin{figure}[H]
\centering
      \includegraphics[width=10cm]{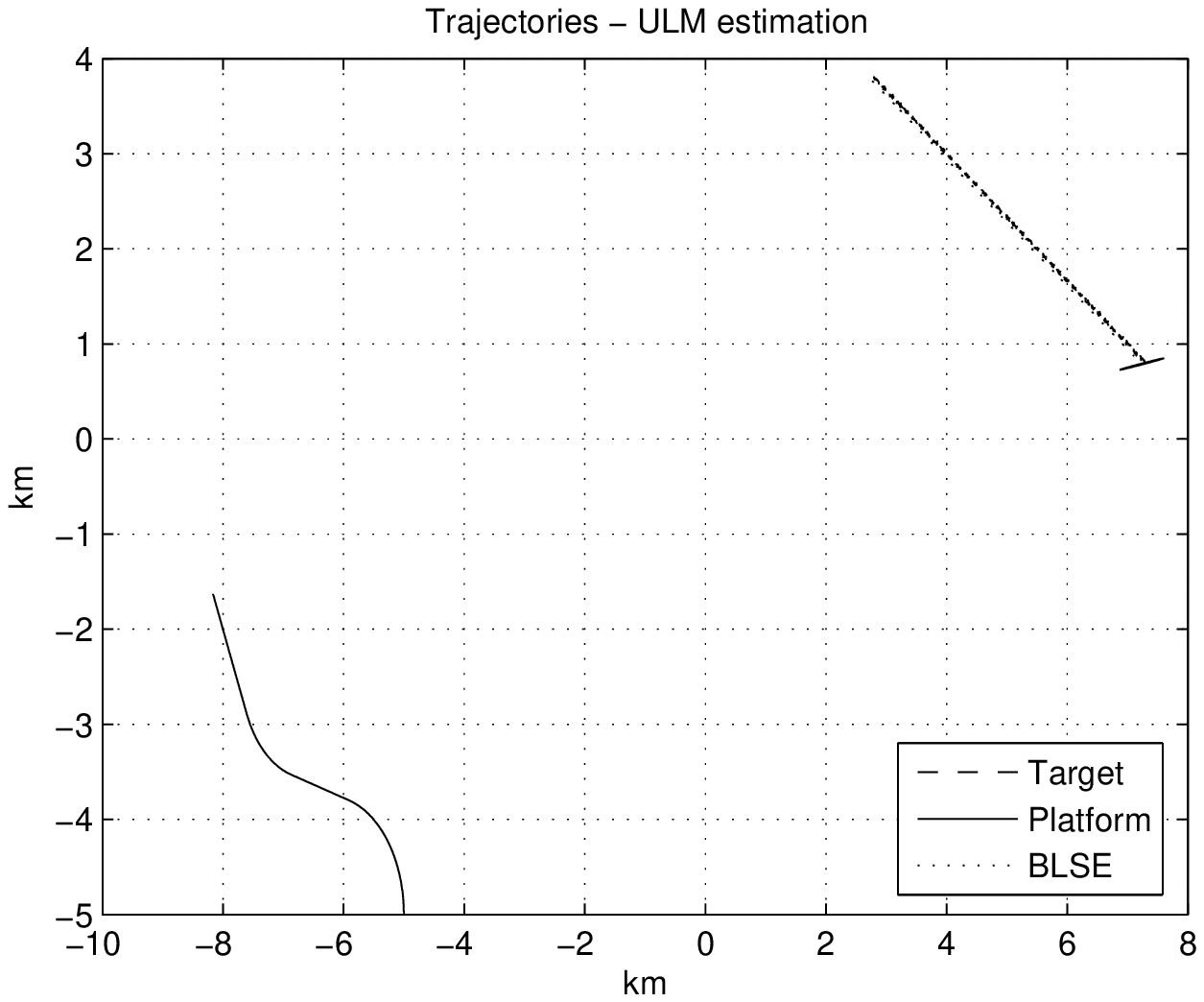}
      \caption{Trajectories with confidence area for BLSE at final position}
      \label{fig:BLSE_isotropic_iid_siml}
\end{figure}

\begin{figure}[H]
   \begin{minipage}[]{.46\linewidth}
      \includegraphics[width=8cm]{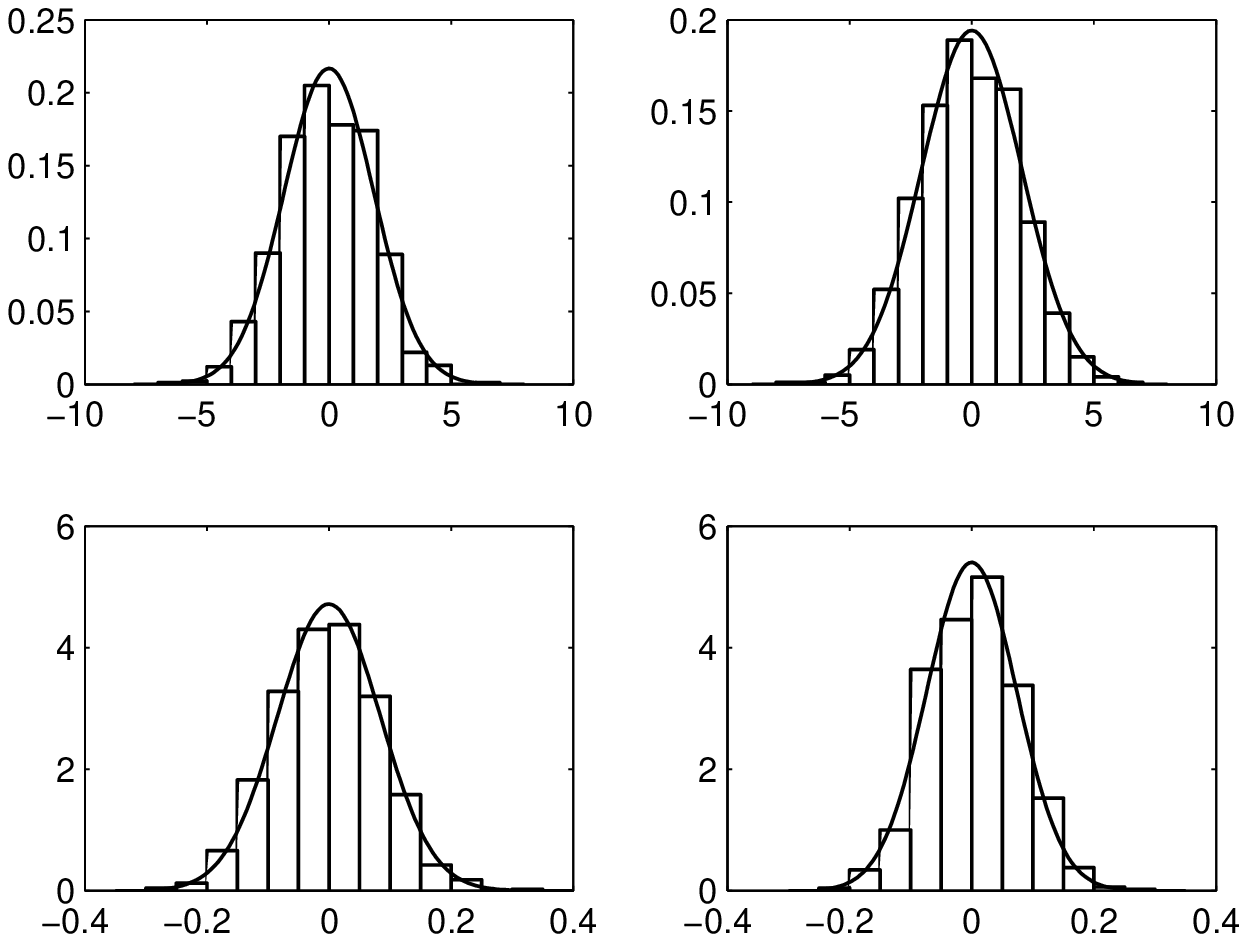}
      \caption{Histograms for BLSE with iid Gaussian isotropic sequence}
      \label{fig:BLSE_isotropic_iid_histo}
   \end{minipage} \hfill
   \begin{minipage}[c]{.46\linewidth}
      \includegraphics[width=8cm]{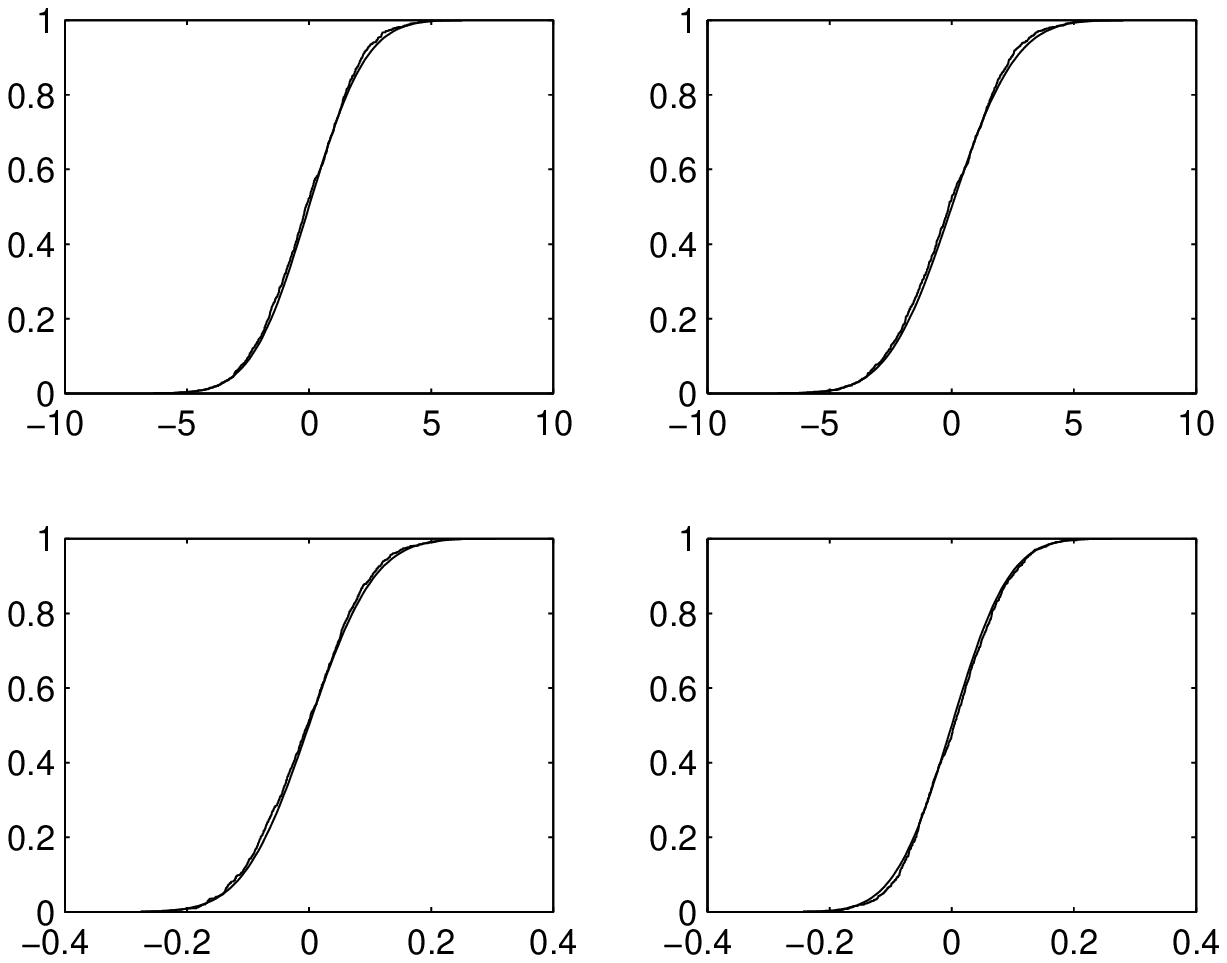}
      \caption{Cumulative distribution functions for BLSE with iid Gaussian isotropic sequence}
      \label{fig:BLSE_isotropic_iid_cdf}
   \end{minipage}
\end{figure}

The figure \ref{fig:MLE_isotropic_iid_histo} present the histograms of the coordinates of $\sqrt{n}(\tilde{\theta}_n-\theta^*)$ with the marginal probability densities of the asymptotic law ${\cal N}(0,I^{-1}(\theta^*))$ in dotted line. Empirical cumulative distribution functions of the coordinates of $\sqrt{n}(\tilde{\theta}_n-\theta^*)$ and marginal cumulative distributions of law ${\cal N}(0,I^{-1}(\theta^*))$ are presented on figure \ref{fig:MLE_isotropic_iid_cdf}. These two figures illustrate the convergence in distribution given by Theorem \ref{theo:lrtpara}.

\begin{figure}[H]
\centering
      \includegraphics[width=10cm]{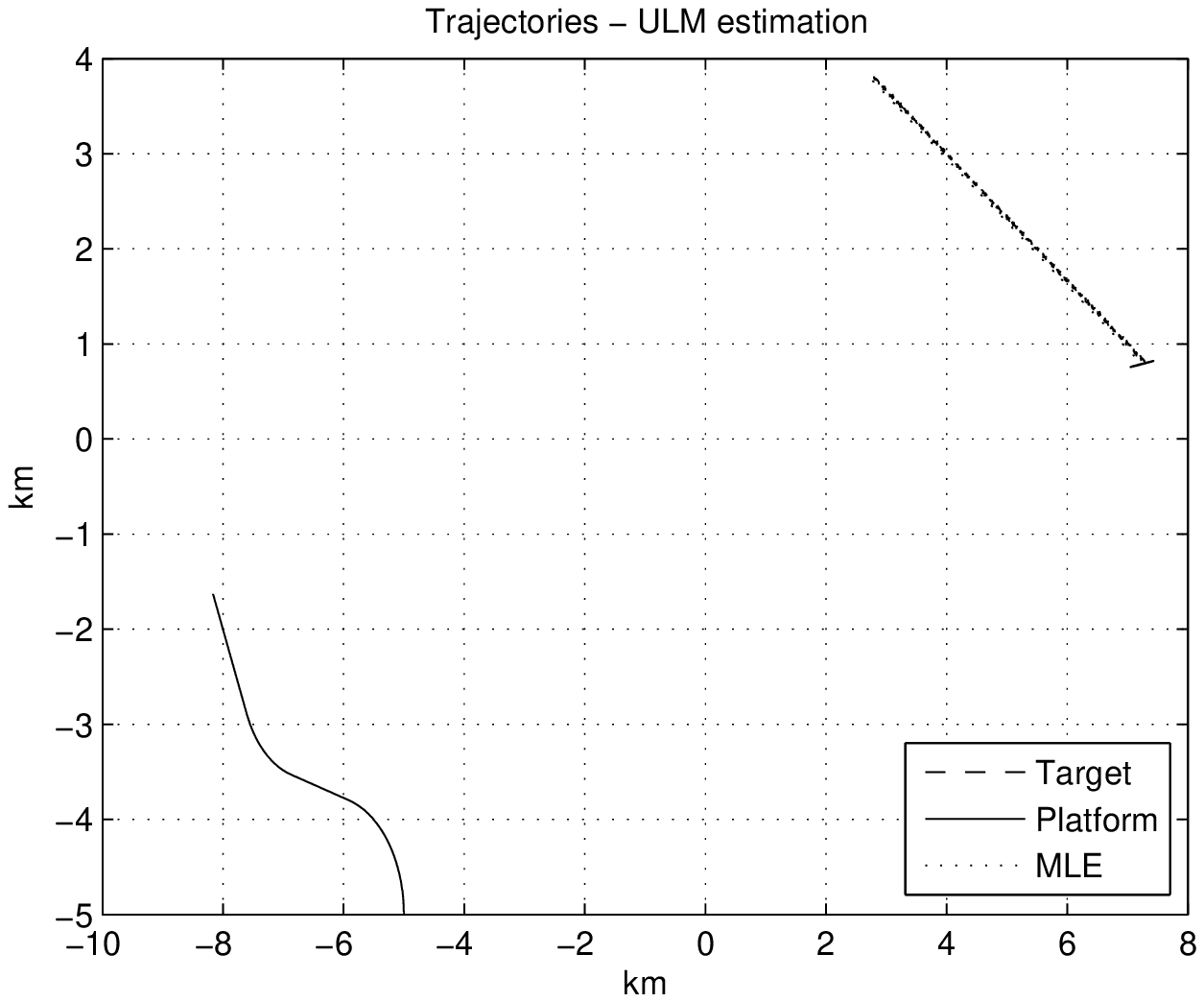}
      \caption{Trajectories with confidence area for MLE at final position}
      \label{fig:MLE_isotropic_iid_siml}
\end{figure}

\begin{figure}[H]
   \begin{minipage}[]{.46\linewidth}
      \includegraphics[width=8cm]{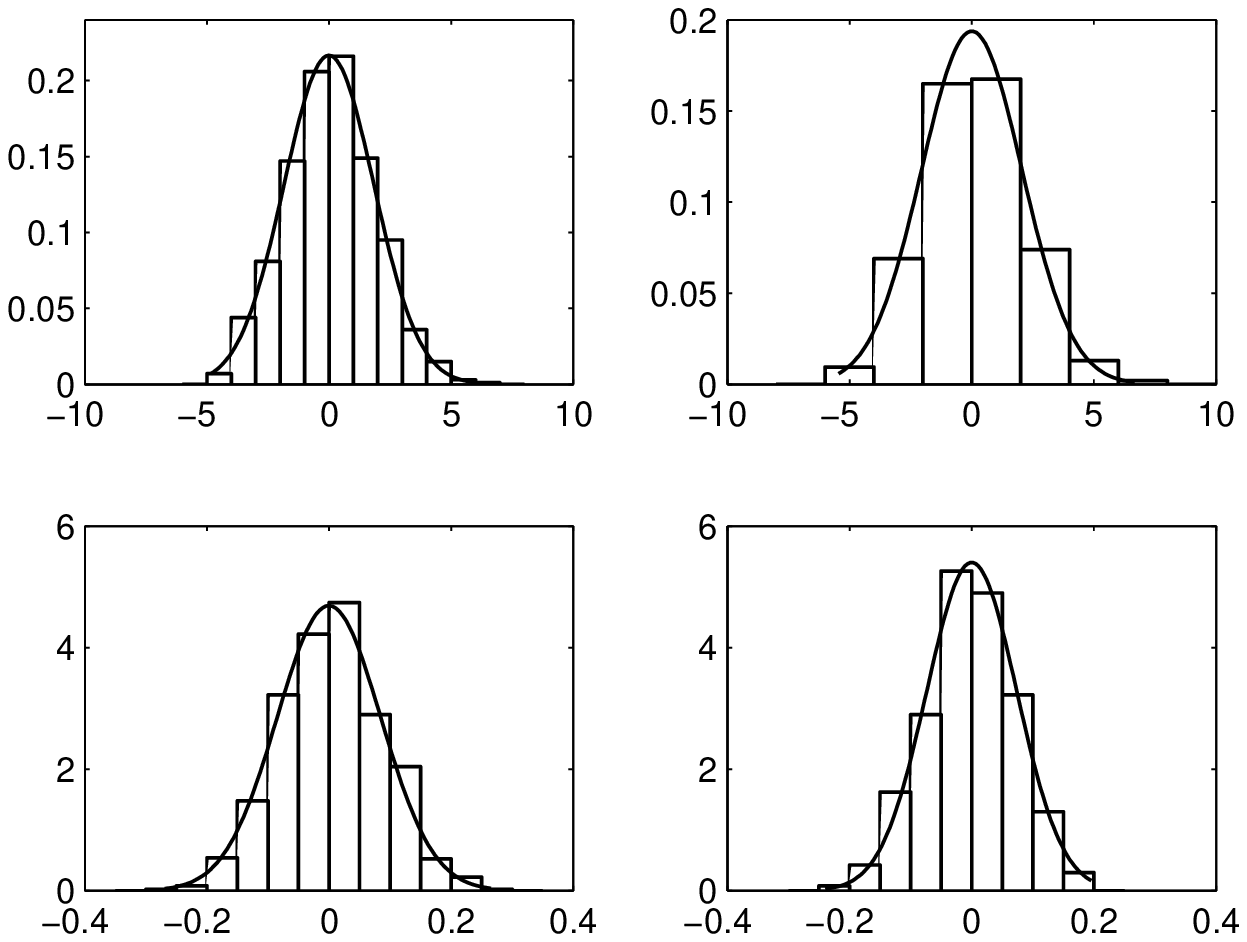}
      \caption{Histograms for MLE with iid Gaussian isotropic sequence}
      \label{fig:MLE_isotropic_iid_histo}
   \end{minipage} \hfill
   \begin{minipage}[c]{.46\linewidth}
      \includegraphics[width=8cm]{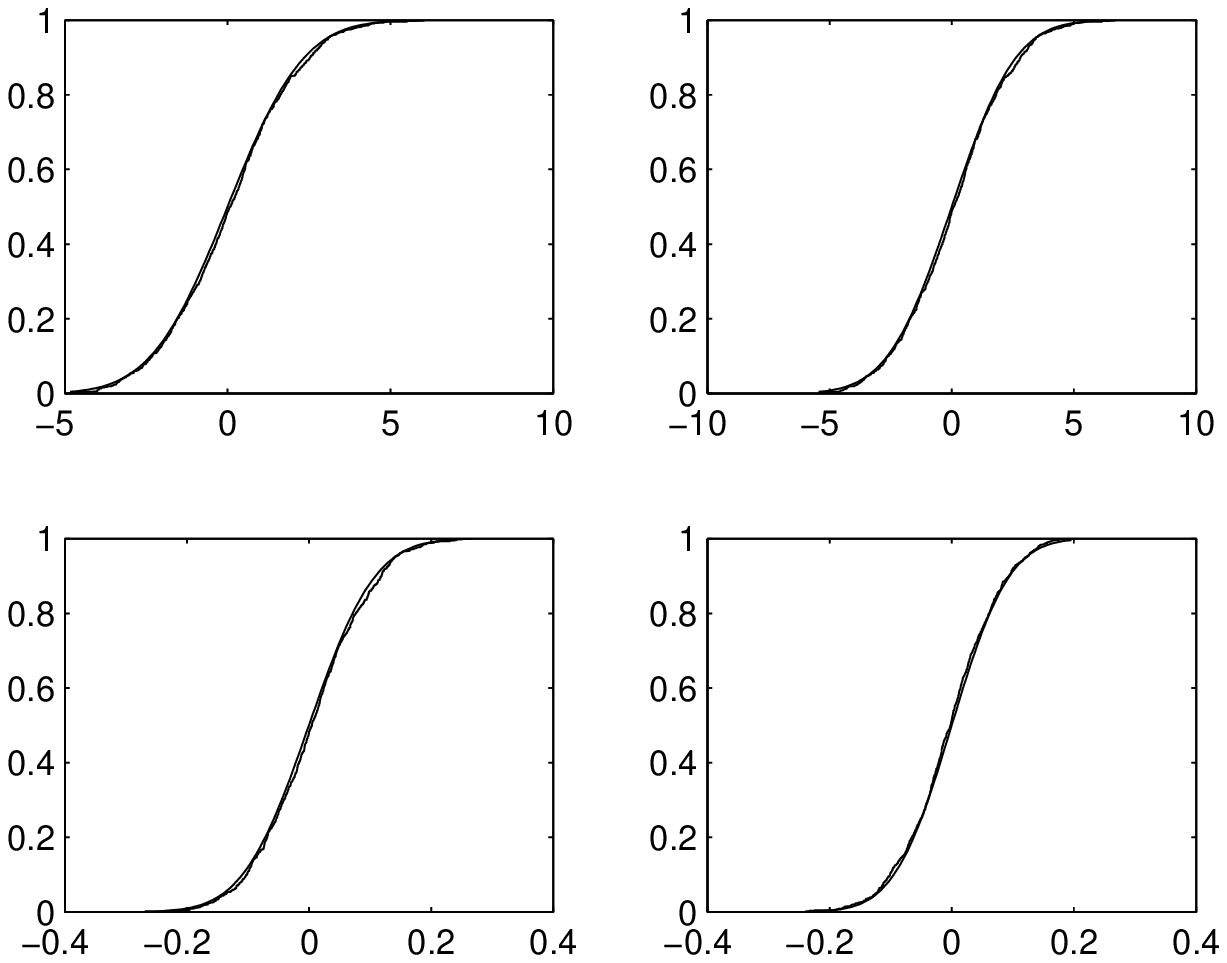}
      \caption{Cumulative distribution functions for MLE with iid Gaussian isotropic sequence}
      \label{fig:MLE_isotropic_iid_cdf}
   \end{minipage}
\end{figure}

Confidence intervals for coordinates of $\theta^*$ with level of $95 \%$ are detailed in table \ref{tab:BLSE_iid_isotropic} for $\overline{\theta}_n$ and in table \ref{tab:MLE_iid_isotropic} for $\tilde{\theta}_n$ and are respectively denoted by ${\rm IC_1}(\overline{\theta}_n)$ and ${\rm IC_3}(\tilde{\theta}_n)$.
We also present in table \ref{tab:CONS_iid_isotropic}  conservative confidence intervals denoted by ${\rm IC_2}(\overline{\theta}_n)$ built on the result provided by
Theorem \ref{theo:BOT:1} with $R_{\min} = 6 \km$. The choice of $\sigma_X$ and $R_{\min}$  is a prior knowledge on the experiment and is made according to the knowledge of the tactical situation of BOT. Note that the majoration obtained in \eqref{eq:conservative_majoration} shows that the accuracy of the conservative confidence intervals is proportional to the ratio $\frac{\esp\| \e_1 \|^2}{R_{\min}^2}$. This result is very interesting in practice since it shows that
for high values of relative distance between target and observer and small values of state noise variance, conservative confidence intervals are of high accuracy.

For these simulations, one needs to calculate $I_\Psi(\overline{\theta}_n)$, $I_\Psi(\theta^*)$, $I(\tilde{\theta}_n)$ and $I(\theta^*)$ which involve expectations of functions of the r.v. $\varepsilon_1$ with law ${\cal N}(O,\sigma_X^2 \times I_2)$. All integrals of this type has been calculated using quadrature formula with 12 points. Abscissas and weight factors are given in \cite{abramowitz:stegun:1964}. Let us detail the numerical values of $I_\Psi(\bar{\theta}_n)$ and $\sigma^2 \times I_R(\overline{\theta}_n)$ for one experiment used to build the estimators $\overline{\theta_n}$ and $\tilde{\theta}_n$. These numerical values illustrate that the contributions of state noise and observation noise are of the same level.
\begin{equation*}
    I_\Psi(\overline{\theta}_n) = 10^{-6} \times \left(
    \begin{array}{rrrr}
    0.0010 &  -0.0014 &   0.0049  & -0.0094\\
   -0.0014 &   0.0024  & -0.0094  &  0.0220\\
    0.0049  & -0.0094 &   0.0400 &  -0.0950\\
   -0.0094  &  0.0220  & -0.0950 &   0.2709
    \end{array}\right) \;,
\end{equation*}
\begin{equation*}
    \sigma^2 \times I_R(\overline{\theta}_n) = 10^{-6} \times \left(
    \begin{array}{rrrr}
    0.0015  & -0.0023  &  0.0082  & -0.0169\\
   -0.0023  &  0.0043  & -0.0169 &   0.0428\\
    0.0082  & -0.0169 &   0.0728  & -0.1853\\
   -0.0169  &  0.0428  & -0.1853 &   0.5639
    \end{array}\right) \;.
\end{equation*}
Let us now precise the values of variance matrices. We have
\begin{equation*}
    I_M^{-1}(\overline{\theta}_n) = \left(
    \begin{array}{rrrr}
    3.4917 &   3.8949 &   0.1560 &  -0.1399\\
    3.8949 &   4.3496 &   0.1752 &  -0.1561\\
    0.1560  &  0.1752 &   0.0074 &  -0.0062\\
   -0.1399 &  -0.1561 &  -0.0062 &   0.0056
    \end{array}\right) \;,
\end{equation*}
and
\begin{equation*}
    I^{-1}(\tilde{\theta}_n)=
    \left(
    \begin{array}{rrrr}
    3.3918  &  3.7884 &   0.1526 &  -0.1359\\
    3.7884  &  4.2362 &   0.1715 &  -0.1518\\
    0.1526  &  0.1715 &   0.0072 &  -0.0061\\
   -0.1359  & -0.1518  & -0.0061 &   0.0055
    \end{array}\right) \;.
\end{equation*}
The true parameter $\theta^*$ is
\begin{equation*}
    \theta^* = \begin{pmatrix} 2.8& 3.8& 0.225& -0.15 \end{pmatrix}^T\;,
\end{equation*}
and values of estimators $\overline{\theta}_n$ and $\tilde{\theta}_n$, used to calculate variance matrices, are
\begin{eqnarray*}
      \overline{\theta}_n &=& \begin{pmatrix}   2.8753 &   3.8841  &  0.2284 &  -0.1530 \end{pmatrix}^T\;, \\
      \tilde{\theta}_n &=& \begin{pmatrix}   2.8067 & 3.8077 & 0.2253 & -0.1502
      \end{pmatrix}^T\;,
\end{eqnarray*}
with $x_0,\ y_0$ given in $\km$ and $v_x,\ v_y$ given in $\kms$ and the position at final time is $(7.3,0.8)$. 
\begin{table}[H]
   \begin{minipage}[]{.46\linewidth}
\begin{tabular}{cc|c}
    \multicolumn{2}{c|}{ $\mathrm{IC}_1(\bar{\theta}_{n,i}$)} & \multicolumn{1}{c}{$| \mathrm{IC}_1(\bar{\theta}_{n,i}$)$|$}\\
  \hline
    7.3128  &  7.5747 &   0.2619\\
    0.8017  &  0.8456 &   0.0439\\
    0.2253  &  0.2316  &  0.0063\\
   -0.1558  & -0.1503 &   0.0055\\
\end{tabular}
\caption{Confidence intervals for BLSE at level 95\%}
\label{tab:BLSE_iid_isotropic}
  \end{minipage} \hfill
   \begin{minipage}[c]{.46\linewidth}
\begin{tabular}{cc|c}
    \multicolumn{2}{c|}{ $\mathrm{IC}_2(\bar{\theta}_{n,i}$)} & \multicolumn{1}{c}{$|\mathrm{IC}_2(\bar{\theta}_{n,i}$)$|$}\\
  \hline
    6.0645 &   8.8230 &   2.7586\\
    0.5917  &  1.0557   & 0.4640\\
    0.1949   & 0.2619  &  0.0669\\
   -0.1818  & -0.1242  &  0.0576\\
\end{tabular}
\caption{Conservative confidence intervals for BLSE at level 95\%}
\label{tab:CONS_iid_isotropic}
   \end{minipage} 
  \center{
\begin{tabular}{cc|c}
    \multicolumn{2}{c|}{ $\mathrm{IC}_3(\tilde{\theta}_{n,i}$)} & \multicolumn{1}{c}{$|$ $\mathrm{IC}_3(\tilde{\theta}_{n,i}$)$|$}\\
  \hline
    7.1842   & 7.4430  &  0.2588\\
    0.7815   & 0.8249  &  0.0434\\
    0.2222   & 0.2285  &  0.0063\\
   -0.1529  & -0.1475  &  0.0054\\
\end{tabular}}
\caption{Confidence intervals for MLE at level 95\%}
\label{tab:MLE_iid_isotropic}
\end{table}
It appears that the maximum likelihood estimator $\tilde{\theta}_n$ is a bit more accurate  than $\overline{\theta}_n$. It is not surprising since the MLE is designed specifically for the model, and takes into account the state noise. Nevertheless, because of the high calculation cost for the MLE, the BLSE is in practice a very useful alternative.

For the second simulation, we consider the case of a sequence $(\varepsilon_k )_{k\in\N}$ of i.i.d Gaussian centered random variables with variance $\sigma_X^2 \times \begin{pmatrix}6^2 & 0 \\ 0 & 1 \end{pmatrix} $ and $\sigma_X = 10 \m$.
It seems that the results given by Theorems \ref{theo:ls:clt} and \ref{theo:lrtpara} still hold,
even though the  sequence $(\varepsilon_k )_{k\in\N}$ does not have an isotropic
distribution,
see Figures \ref{fig:BLSE_nonisotropic_iid_histo}, \ref{fig:BLSE_nonisotropic_iid_cdf}, \ref{fig:MLE_nonisotropic_iid_histo} and \ref{fig:MLE_nonisotropic_iid_cdf}.
\begin{figure}[H]
   \begin{minipage}[]{.46\linewidth}
      \includegraphics[width=8cm]{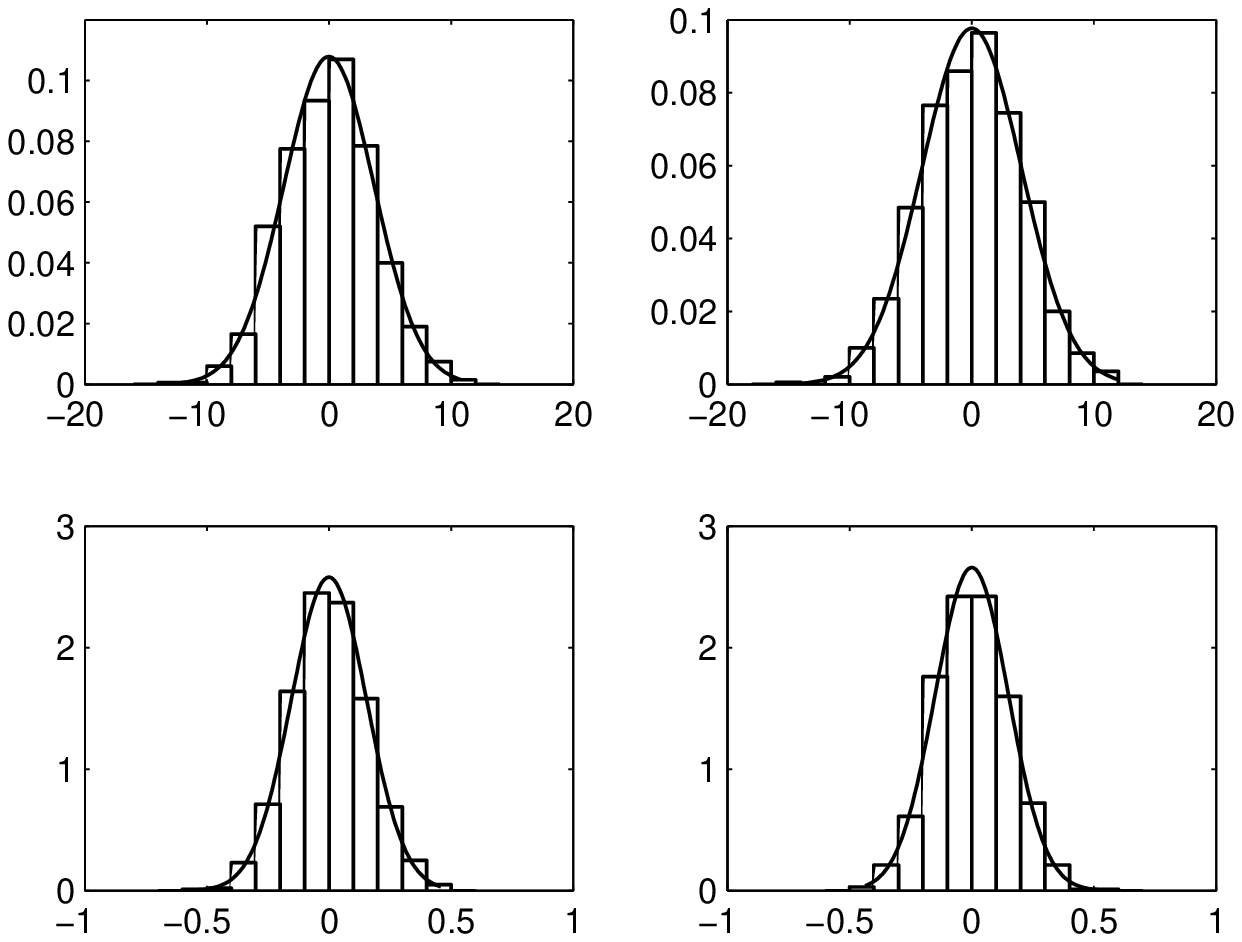}
      \caption{Histograms for BLSE with iid Gaussian non-isotropic sequence}
      \label{fig:BLSE_nonisotropic_iid_histo}
   \end{minipage} \hfill
   \begin{minipage}[c]{.46\linewidth}
      \includegraphics[width=8cm]{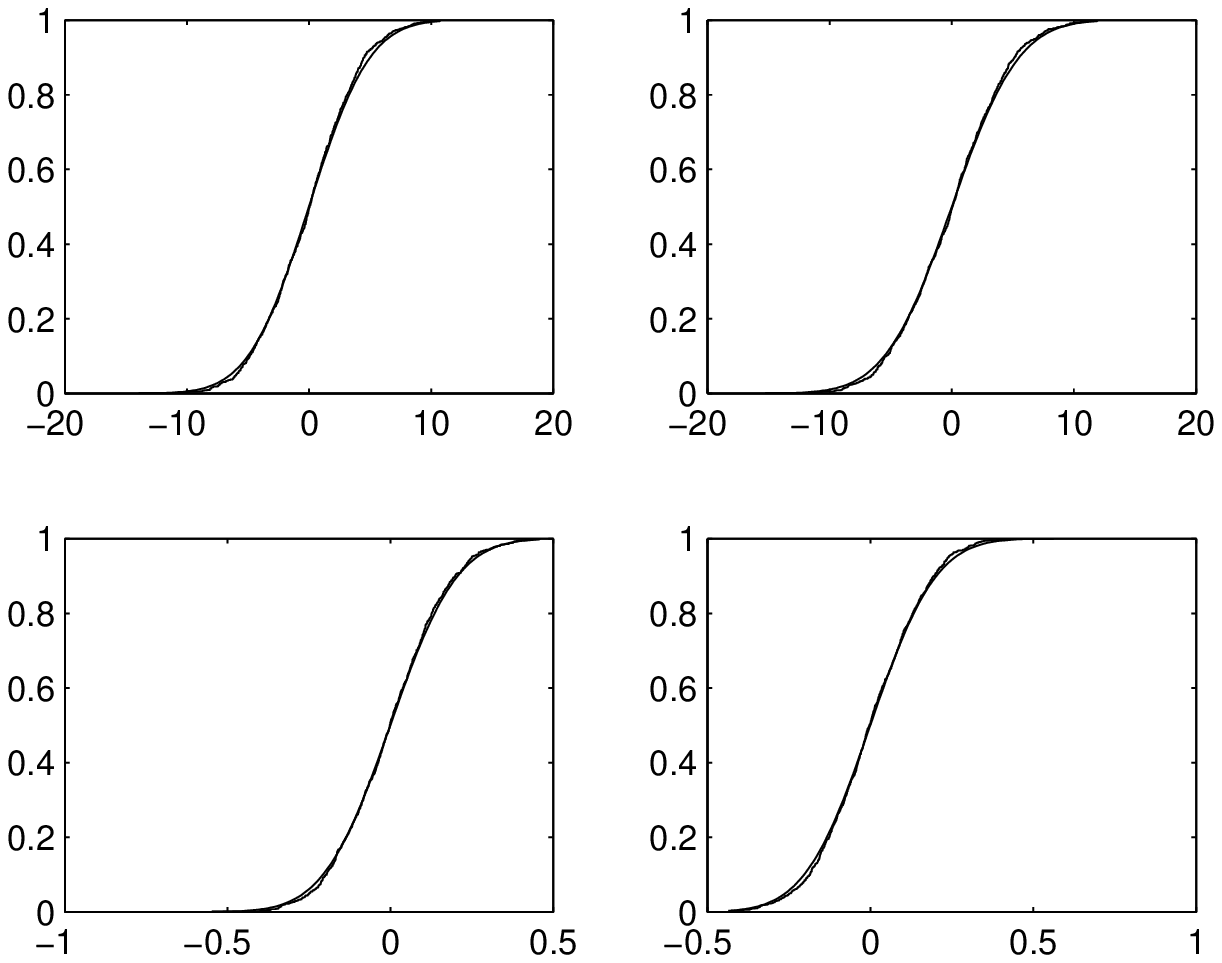}
      \caption{Cumulative distribution functions for BLSE with iid Gaussian non-isotropic sequence}
      \label{fig:BLSE_nonisotropic_iid_cdf}
   \end{minipage}
\end{figure}
\begin{figure}[H]
   \begin{minipage}[]{.46\linewidth}
      \includegraphics[width=8cm]{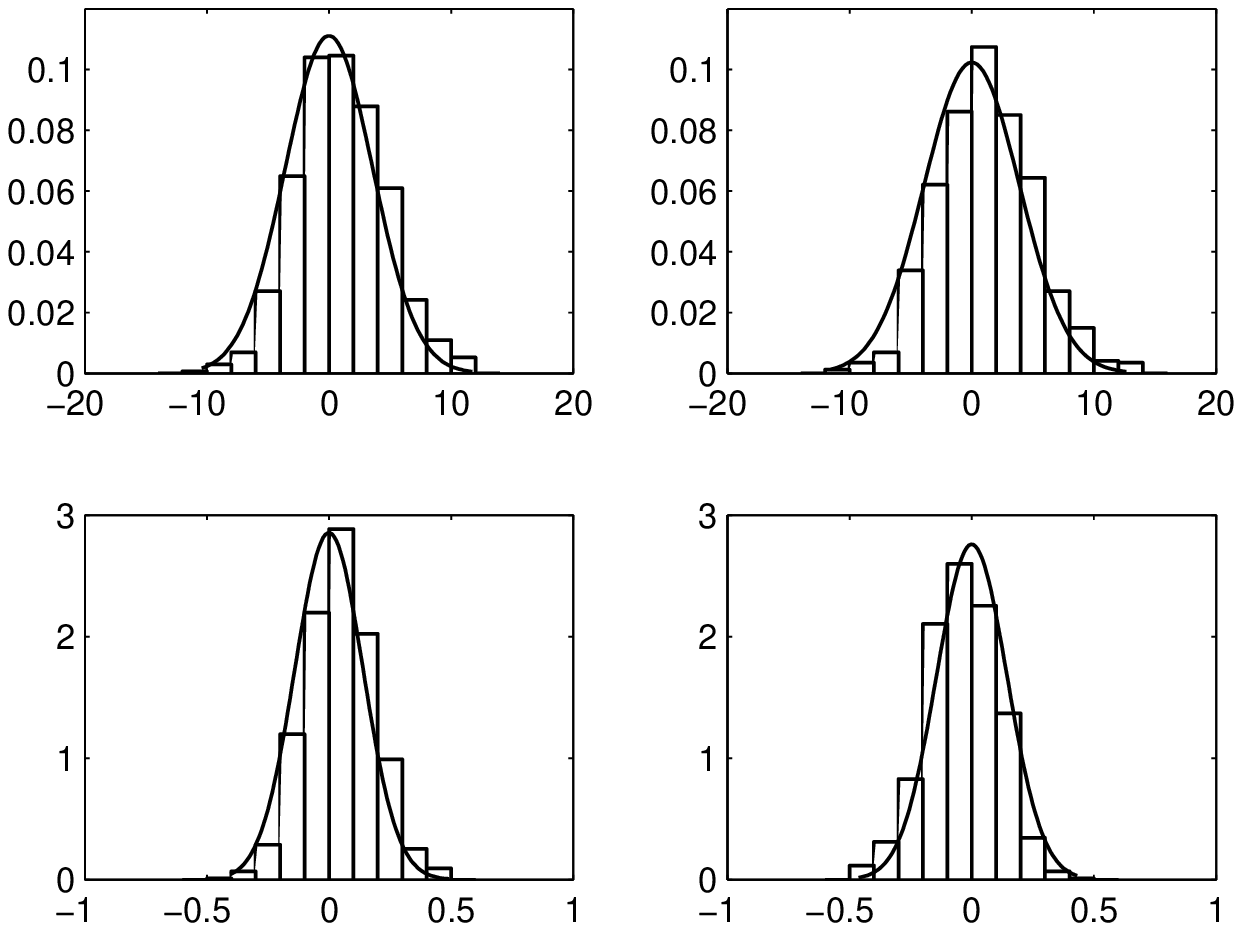}
      \caption{Histograms for MLE with iid Gaussian non-isotropic sequence}
      \label{fig:MLE_nonisotropic_iid_histo}
   \end{minipage} \hfill
   \begin{minipage}[c]{.46\linewidth}
      \includegraphics[width=8cm]{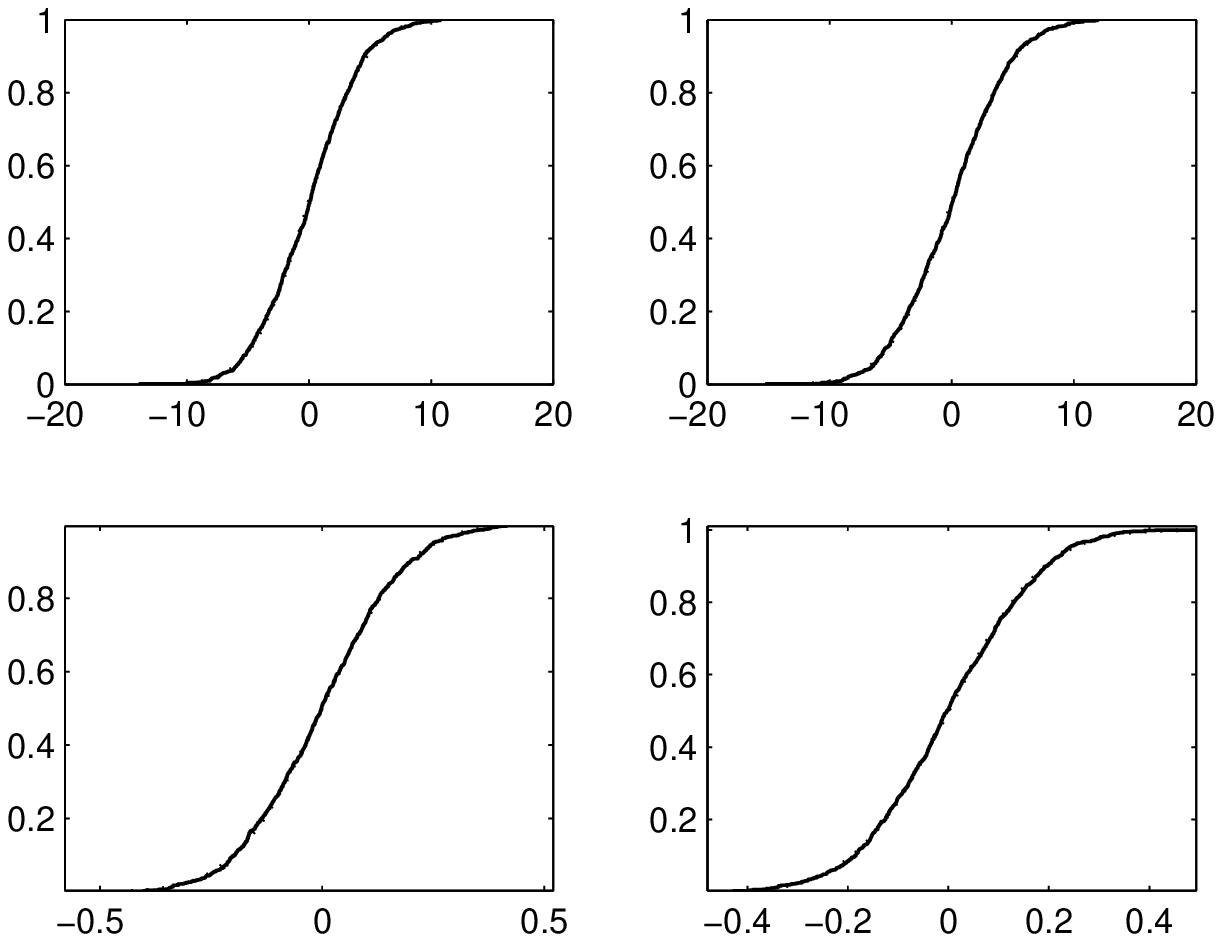}
      \caption{Cumulative distribution functions for MLE with iid Gaussian non-isotropic sequence}
      \label{fig:MLE_nonisotropic_iid_cdf}
   \end{minipage}
\end{figure}
The estimators  values are
\begin{eqnarray*}
      \bar{\theta}_n &=& \begin{pmatrix} 2.8383 & 3.8440 &  0.2264& -0.1516 \end{pmatrix}^T\;, \\
      \tilde{\theta}_n &=& \begin{pmatrix}      2.7984 &  3.7999 & 0.2253 & -0.1499
      \end{pmatrix}^T\;,
\end{eqnarray*}

The values of variance matrices for the two estimators are
\begin{equation*}
    I_M^{-1}(\bar{\theta}_n) = \left(
    \begin{array}{rrrr}
   15.4505 &  17.0122&    0.6174&   -0.6253\\
   17.0122 &  18.7661 &   0.6863 &  -0.6889\\
    0.6174  &  0.6863 &   0.0263 &  -0.0250\\
   -0.6253 &  -0.6889 &  -0.0250 &   0.0253
    \end{array}\right) \;,
\end{equation*}
and
\begin{equation*}
    I^{-1}(\tilde{\theta}_n) = \left(
    \begin{array}{rrrr}
   12.9538  & 14.0399 &   0.4766 &  -0.5214\\
   14.0399  & 15.2720 &   0.5262 &  -0.5661\\
    0.4766   & 0.5262 &   0.0197  & -0.0192\\
   -0.5214  & -0.5661  & -0.0192  &  0.0210
    \end{array}\right) \;,
\end{equation*}
The confidence intervals detailed in table \ref{tab:BLSE_iid_nonisotropic} and table \ref{tab:MLE_iid_nonisotropic} show that the maximum likelihood estimator $\tilde{\theta}_n$ is significantly more accurate than the BLSE. Comparing to the first simulation where the difference is not so large, the higher accuracy of $\tilde{\theta}_n$ can be understood because of the higher level state noise in this simulation. Then, taking into account this state noise for estimating the parameter provides a significantly better result. The conservative intervals for  $R_{\min} = 6 \km$ described in table \ref{tab:CONS_iid_nonisotropic} are quite large compared to those obtained for the first simulation. This inaccuracy results directly from the large value of $\esp \| \e_1\|^2$ chosen for the state noise.

\begin{table}[H]
   \begin{minipage}[]{.46\linewidth}
\begin{tabular}{cc|c}
    \multicolumn{2}{c|}{ $\mathrm{IC}_1(\bar{\theta}_{n,i}$)} & \multicolumn{1}{c}{$|$ $\mathrm{IC}_1(\bar{\theta}_{n,i}$)$|$}\\
  \hline
    7.1040  &  7.6275 &   0.5235\\
    0.7698  &  0.8552 &   0.0854\\
    0.2204  &  0.2323 &   0.0119\\
   -0.1574 &  -0.1457  &  0.0117
\end{tabular}
\caption{Confidence intervals for BLSE at level 95\%}
\label{tab:BLSE_iid_nonisotropic}
   \end{minipage} \hfill
   \begin{minipage}[c]{.46\linewidth}
\begin{tabular}{cc|c}
    \multicolumn{2}{c|}{ $\mathrm{IC}_2(\bar{\theta}_{n,i}$)} & \multicolumn{1}{c}{$|$ $\mathrm{IC}_2(\bar{\theta}_{n,i}$)$|$}\\
  \hline
    1.5049  & 13.2266 &  11.7218\\
   -0.1740  &  1.7990  &  1.9730\\
    0.0842 &   0.3686 &   0.2844\\
   -0.2740  & -0.0291  &  0.2449\\
\end{tabular}
\caption{Conservative confidence intervals for  BLSE at level 95\%}
\label{tab:CONS_iid_nonisotropic}
   \end{minipage}
\center{
\begin{tabular}{cc|c}
    \multicolumn{2}{c|}{ $\mathrm{IC}_3(\tilde{\theta}_{n,i}$)} & \multicolumn{1}{c}{$|$ $\mathrm{IC}_3(\tilde{\theta}_{n,i}$)$|$}\\
  \hline
    7.0721  &  7.5366 &   0.4645\\
    0.7643 &   0.8388  &  0.0746\\
    0.2201  &  0.2305  &  0.0103\\
   -0.1552  & -0.1446  &  0.0107\\
\end{tabular}
\caption{Confidence intervals for MLE at level 95\%}
\label{tab:MLE_iid_nonisotropic} }
\end{table}

For the third and last simulation, the  sequence $(\varepsilon_k )_{k\in\N}$ is an AR(1) series such that
\begin{equation*}
    \forall k \in \N \qquad \varepsilon_{k+1} = \Phi \varepsilon_k + \eta_k\;,
\end{equation*}
where $ \Phi = 0.6$ and $( \eta_k)_{k \in \N}$ is a sequence of i.i.d. random variables with
law ${\cal N}(0,\sigma_\eta^2)$ and $\sigma_\eta=8 \m$. Thus, the sequence of state noise $(\varepsilon_k)_{k \in \N}$ is a dependent stationary sequence such that the mixing coefficient $\alpha_k$  tends exponentially fast to zero as $k$ tends to infinity. Then, we observe the predicted behavior described by Proposition \ref{prop:2}. Indeed, by drawing the densities and cumulative distribution functions of the centered Gaussian law  with the empirical variance, we observe a very good adequacy to the Gaussian behavior, see figures \ref{fig:BLSE_AR1_Gauss_histo} and \ref{fig:BLSE_AR1_Gauss_cdf}.

\begin{figure}[H]
   \begin{minipage}[]{.46\linewidth}
      \includegraphics[width=8cm]{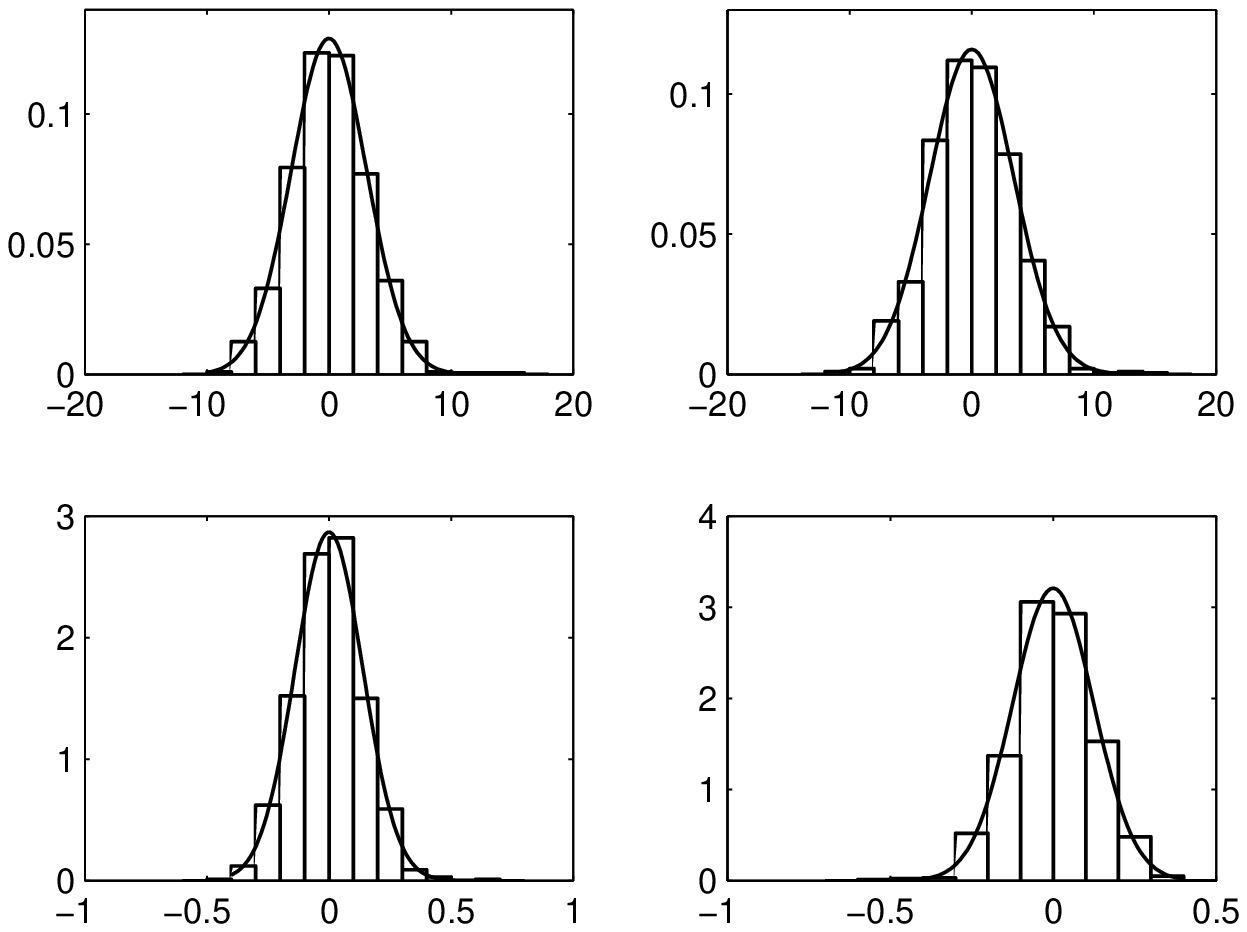}
      \caption{Histograms for AR(1) sequence, Gaussian adequacy}
      \label{fig:BLSE_AR1_Gauss_histo}
   \end{minipage} \hfill
   \begin{minipage}[c]{.46\linewidth}
      \includegraphics[width=8cm]{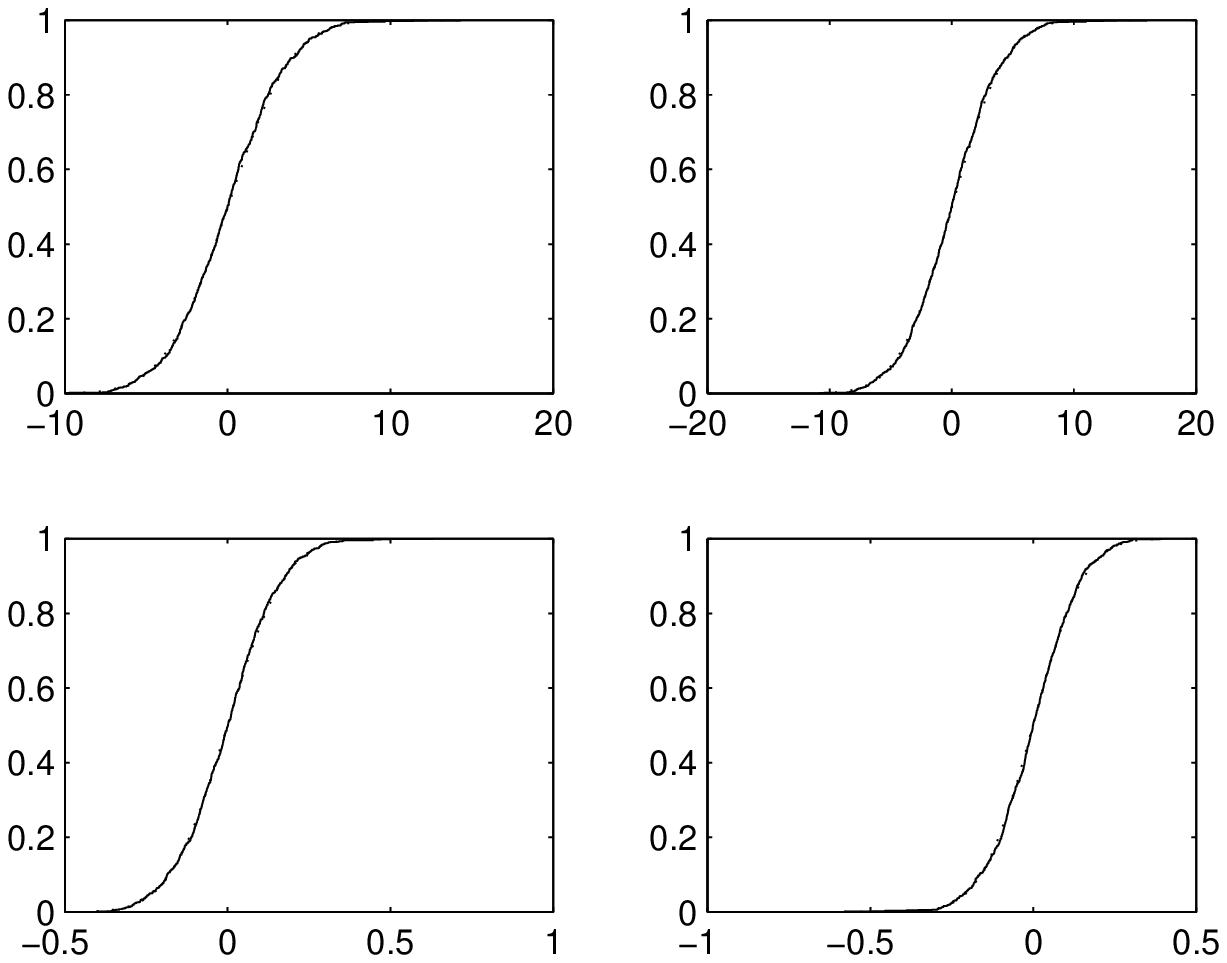}
      \caption{Cumulative distribution functions for AR(1) sequence,  Gaussian adequacy}
      \label{fig:BLSE_AR1_Gauss_cdf}
   \end{minipage}
\end{figure}

$\\$

{\em Acknowledgements:} the authors want to thank Jer\^{o}me Dedecker for helpful discussions about
dependent sequences of random variables.

Elisabeth Gassiat,  Laboratoire de Math\'ematique,   Universit\'e Paris-Sud 11,
B\^atiment 425,  91~405 Orsay C\'edex,  France.
E-mail:  elisabeth.gassiat@math.u-psud.fr\\ \\
Beno\^it Landelle,  Laboratoire de Math\'ematique,   Universit\'e Paris-Sud 11,
B\^atiment 425,  91~405 Orsay C\'edex,  France.
E-mail:  benoit.landelle@math.u-psud.fr


\begin{thebibliography}{}

\bibitem{abramowitz:stegun:1964}
Abramowitz, Milton and Stegun, Irene A. (1964).
\newblock {\em Handbook of mathematical functions with formulas, graphs, and mathematical tables}
\newblock National Bureau of Standards Applied Mathematics Series.

\bibitem{barshalom:li:kirubarajan:2001}
Bar-Shalom, Y., Rong Li, X. \&   Kirubarajan, T. (2001).
\newblock {\em Estimation with Applications to Tracking and Navigation}
\newblock Wiley-Interscience.

\bibitem{castillo:2008}
Castillo, I. (2008).
\newblock A semi-parametric Bernstein-von Mises theorem.
\newblock {\em submitted}.

\bibitem{doucet:defreitas:gordon:2001}
Doucet, A., de Freitas, N. \& Gordon, N. (2001).
\newblock {\em Sequential Monte Carlo Methods in Practice}.
\newblock Springer.

\bibitem{Ibragimov:1962}
Ibragimov, I.A. (1962).
\newblock Some limit theorems for stationary processes.
\newblock {\em Theory Probab. Appl.\/} {\bf 7}, 349--382.

\bibitem{landelle:2008}
Landelle, B. (2008).
\newblock Robustness considerations for bearings only tracking.
\newblock {\em Proceedings of the 11th International Conference on Information Fusion (FUSION 2008), Cologne, Germany}.

\bibitem{landelle}
Landelle, B. (2008).
\newblock Etude statistique du probl\`eme de la trajectographie passive.
\newblock {\em Th\`ese de l'Universit\'e Paris-Sud, manuscript}.

\bibitem{lecam:1986}
Le Cam, L. (1986).
\newblock {\em Asymptotic methods in statistical decision theory}.
\newblock New-York, Springer-Verlag.

\bibitem{mcneney:wellner:2000}
Mc Neney, B.  \& Wellner, J.A. (2000).
\newblock Application of convolution theorems in semiparametric models with non i.i.d. data.
\newblock {\em Journal of Statistical Planning and Inference\/} {\bf 91}, 441--480.

\bibitem{averbuch:barshalom:dayan:mazor:1998}
Mazor, E. , Averbuch, A.,  Bar-Shalom, Y. \& Dayan, J. (1998).
\newblock Interacting Multiple Model Methods in Target Tracking: A Survey.
\newblock {\em IEEE Trans. Aerosp. Electron. Syst.} {\bf  34},  1,  103-123.

\bibitem{rio:1995}
Rio, E. (1995).
\newblock About the Lindeberg method for strongly mixing sequences.
\newblock {\em ESAIM Probability and Statsitics} {\bf 1}, 35-61.

\bibitem{rio:2000}
Rio, E. (2000).
\newblock  {\em Th\'eorie asymptotique des processus al\'eatoires faiblement d\'ependants}
\newblock Math\'ematiques et Applications, Springer.


\bibitem{arulampalam:gordon:ristic:2004}
Ristic, B.,  Arulampalam, S. \&  Gordon, N. (2004).
\newblock {\em Beyond the Kalman Filter}
\newblock Artech House.

\bibitem{vandervaart:1998}
Van der Vaart, A. (1998).
\newblock  {\em Asymptotic Statistics}
\newblock Cambridge University Press.

\end{thebibliography}
 \end{document}